\documentclass[11pt]{amsart}

\pdfoutput=1

\usepackage[margin=0.25in]{geometry}

\usepackage{color}
\usepackage{geometry}
\usepackage{latexsym}
\usepackage{amssymb}
\usepackage{amsthm}
\usepackage{amscd}
\usepackage{amsmath}
\usepackage{mathrsfs}
\usepackage{tikz}
\usepackage{tikz-cd}
\usepackage{tkz-fct} 
\usepackage{mathabx}
\usepackage{stmaryrd}
\usepackage{listings}
\usepackage{youngtab}
\usepackage{pgfplots}
\usepackage{rotating}
\usetikzlibrary{shapes.geometric,positioning}
\usepackage{hyperref}
\usepackage{adjustbox}
\usepackage{tikz-3dplot}
\usepgflibrary{arrows}
\usepackage{graphicx}
\usetikzlibrary{calc}
\usepackage{amsaddr}

\tdplotsetmaincoords{60}{115}
\pgfplotsset{compat=newest}

\newtheorem{theorem}{Theorem}[section]

\newtheorem{lemma}[theorem]{Lemma}
\newtheorem{definition}[theorem]{Definition}
\newtheorem{proposition}[theorem]{Proposition}

\newtheorem{conjecture}[theorem]{Conjecture}

\newtheorem{example}[theorem]{Example}
\newtheorem{remark}[theorem]{Remark}

\theoremstyle{definition}

\newenvironment{psmallmatrix}
{\left(\begin{smallmatrix}}
	{\end{smallmatrix}\right)}

\numberwithin{equation}{section}

\makeatletter
\renewcommand{\@makefnmark}{\mbox{\textsuperscript{}}}
\makeatother

\allowdisplaybreaks[1]

\newcommand{\Proj}{\operatorname{Proj}} 	
\newcommand{\GL}{\mathrm{GL}} 	
\newcommand{\SL}{\mathrm{SL}} 	
\newcommand{\Isom}{\mathrm{Isom}} 	
\newcommand{\Aut}{\mathrm{Aut}} 	
\newcommand{\End}{\mathrm{End}} 	
\newcommand{\Gal}{\mathrm{Gal}} 	
\newcommand{\Spec}{\operatorname{Spec}} 	
\newcommand{\cC}{\mathcal{C}}		
\newcommand{\cH}{\mathcal{H}}		
\newcommand{\cM}{\mathcal{M}}		
\newcommand{\cN}{\mathcal{N}}		
\newcommand{\cO}{\mathcal{O}}		
\newcommand{\cU}{\mathcal{U}}		

\newcommand{\sF}{\mathscr{F}}		
\newcommand{\sL}{\mathscr{L}}		
\newcommand{\sT}{\mathscr{T}}		
\newcommand{\sX}{\mathscr{X}}		
\newcommand{\sY}{\mathscr{Y}}		

\newcommand{\bbC}{\mathbb{C}}		
\newcommand{\bbF}{\mathbb{F}}		
\newcommand{\bbG}{\mathbb{G}}		
\newcommand{\bbK}{\mathbb{K}}		
\newcommand{\bbP}{\mathbb{P}}		
\newcommand{\bbQ}{\mathbb{Q}}		
\newcommand{\bbR}{\mathbb{R}}		
\newcommand{\bbZ}{\mathbb{Z}}		

\newcommand{\Mod}[1]{\ (\mathrm{mod}\ #1)}

\newcommand{\ignore}[1]{}

\begin{document}
	
\title{The Geometry of Drinfeld Modular Forms}
\author{Jesse Franklin}
\email{\href{mailto:jfranklin1185@gmail.com}{jfranklin1185@gmail.com}}
	
	\begin{abstract}
			We give a geometric perspective on the algebra of Drinfeld modular forms for congruence subgroups $\Gamma\leq \GL_2(\bbF_q[T]).$ In particular, we describe an isomorphism between the section ring of a line bundle on the stacky modular curve for $\Gamma_2$ and the algebra of Drinfeld modular forms for $\Gamma_2,$ where $\Gamma_2$ is the subgroup of square-determinant matrices in $\Gamma.$ This allows one to compute the latter ring by geometric invariants using the techniques of Voight, Zureick-Brown and O'Dorney. We also show how to decompose the algebra of modular forms for $\Gamma_2$ into a direct sum of two algebras of modular forms for $\Gamma$ and generalize this result to a larger class of congruence subgroups. 
		\end{abstract}
		\maketitle
		\tableofcontents
		
		\section{Introduction}
		
		\subsection{History and Motivation}
		The theory of modular forms in the classical number-field case has existed since the $1800$'s. It is well-understood that modular forms are sections of a particular line bundle on some stacky modular curve. In this set up the geometry of the stacks, with tools such as the Riemann-Roch theorem for stacky curves for example, can be used to compute section rings which describe algebras of modular forms. The program of \cite{VZB} for computing the canonical ring of log stacky curves in all genera even gives minimal presentations for many such section rings, that is: explicit generators and relations, which correspond to generators and relations for algebras of modular forms.\\
		
		In his $1986$ monograph \cite[Page $\mathrm{XIII}$]{Gekeler-Curves} asks for a description of algebras of Drinfeld modular forms in terms of generators and relations. The main results of this note describe the geometry of those modular forms, which allows one to employ techniques such as those in \cite{VZB} to find the desired generators and relations by considering the geometry of the corresponding Drinfeld modular curve. That is, we provide a means to address Gekeler's problem via geometric invariants.\\
		
		Until now, the closest analogies in the Drinfeld setting for the isomorphism $f\mapsto f(z)(dz)^{\otimes k/2}$ between modular forms of weight $k$ and sections of a line bundle on the modular curve from the classical setting are \cite[Proposition $5.6$]{Bockle-EichlerShimura} and \cite[Theorem $5.4$]{Gekeler-Curves} for rank $2$ Drinfeld modules, and \cite[Definition $(10.1)$]{Basson-Breuer-Pink-part2} for the more difficult case of Drinfeld modules of general rank. In \cite[Lemma $(10.7)$]{Basson-Breuer-Pink-part2} there is also an isomorphism between a ring of modular forms and a section ring of form $f(z)\mapsto f(z)dz^{\otimes k/2}.$ Note that a manuscript version of the three preprints by Basson-Breuer-Pink on Drinfeld modular forms of arbitrary rank is due to appear in Memoirs of the AMS, so our citation here will soon not be the most recent version of their theory.\\ 
		
		There is a collection of results which is similar to our work in comparing modular forms for various congruence subgroups to each other as in our second main result Theorem \ref{thm: decomp of mod forms}. Pink finds isomorphisms between algebras of Drinfeld modular forms for open compact subgroups $K\leq \GL_r(\widehat{\bbF_q[T]}),$ where the hat symbol denotes the pro-finite completion $\widehat{\bbF_q[T]}=\prod_{\mathfrak{p}} (\bbF_q[T])_{\mathfrak{p}},$ and normal subgroups $K'\lhd K$ in e.g.\ \cite[Proposition $5.5$]{Pink-compactification-Drinfeld-modular-varieties-2012}. Pink also describes Drinfeld modular forms as sections of an invertible sheaf in \cite[Section $5$]{Pink-compactification-Drinfeld-modular-varieties-2012} which is similar to Theorem \ref{thm: forms to differentials}. However, Pink needs the dual of the relative Lie algebra over a line bundle, rather than the bundle itself, to describe Drinfeld modular forms, which is a major difference between our work. Pink also deals with Drinfeld modules of arbitrary rank while we focus on rank $2$ only, which explains this difference in machinery.\\
		
		There are also some existing results which approach Gekeler's problem, such as Cornelissen's papers \cite{Cornelissen-lvlT} and \cite{Cornelissen-wt1} which deal with linear level moduli spaces (\cite[Theorem $(3.3)$]{Cornelissen-wt1}), i.e.\ the algebra of modular forms for $\Gamma(\alpha T+\beta),$ and include some results for quadratic level (\cite[Proposition $(3.4)$]{Cornelissen-wt1}). 
		Another example, \cite[Theorem $(4.4)$]{Dalal-Kumar-Gamma_0(T)-structure}, computes the algebra of Drinfeld modular forms for $\Gamma_0(T).$ There is also a well-known correspondence between $M^2_{2,1}(\Gamma(N))$ and holomorphic differentials on $X_0(N)$ found in \cite[Section $2.10$]{Gekeler-jacobians}.\\
		
		Several ideas in \cite{Breuer-Gekeler-h-function} are central to our argument, as well as being an exposition on aspects of Gekeler's problem in general. In particular, \cite{Breuer-Gekeler-h-function} introduces the subgroup $\Gamma_2$ of a given congruence subgroup $\Gamma\leq\GL_2(A)$ and gives a moduli interpretation of the corresponding Drinfeld modular curve. Even by the date of these most recent papers, the generalization to the algebra of modular forms for $\Gamma_0(N)$ for any level $N,$ all subgroups $\Gamma_1(N),$ high level (i.e.\ $\deg(N)\geq 2$) $\Gamma(N)$ examples and congruence subgroups of $\SL_2(\bbF_q[T])$ seem to be wide open.\\
		
		Our work differs considerably from the papers from Basson, B\"ockle, Breuer, Cornelissen, Dalal-Kumar, Gekeler, Pink and Reversat cited above in that we work with Drinfeld moduli stacks as opposed to schemes. As early as \cite{Gekeler-Curves} and \cite{Laumon-cohomology-Drinfeld-modular-varieties} it was known that moduli of Drinfeld modules of fixed rank are Deligne-Mumford stacks, but it is the more recent results of \cite{VZB} for computing log canonical rings of stacky curves, and \cite{Porta-Yu-Higher-analytic-stacks-GAGA} which provides a crucial principle of rigid analytic GAGA (short for ``g\'eom\'etrie alg\'ebrique et g\'eom\'etrie analytique'') for stacks, that makes our work possible.\\
		
		There is some historical reason for our choice to work with rigid analytic spaces as opposed to the more general adic or Berkovich spaces, namely the original analytic theory of the Drinfeld setting was developed in that language in e.g.\ Goss's paper \cite{Goss-upper-halfplane}. Though there is for example a more general or modern theory of adic stacks (see e.g.\ \cite{Warner-thesis-adic-moduli-spaces}) we will find it more convenient to phrase things in terms of rigid analytic spaces.
		
		\subsection{Main Results}
		
		This article describes the geometry of Drinfeld modular forms: we associate to each Drinfeld modular form a section of a particular line bundle on a specified stacky modular curve.
		We also give a decomposition of the algebra of modular forms, which allows one to compute all of the section rings in the papers above by means of the geometric techniques of \cite{VZB}. This means we have Gekeler's elementary interpretation of the generating modular forms in terms of Drinfeld modules viewed as points of the moduli space. We make no restrictions on the level, and when we insist that our congruence subgroup in question contains diagonal matrices this is only to simplify proofs. So, we demonstrate a new way to compute algebras of modular forms in great generality, and in a way which makes the problem reliant only on the geometry of the modular curve.\\  
		
		Our technique relies on the following three theorems which we quickly set up before stating.\\
		
		Let $\Gamma$ be a congruence subgroup of $\GL_2(\bbF_q[T]$ and suppose that
		$\displaystyle{\det(\Gamma)\overset{def}{=}\{\det(\gamma):\gamma\in \Gamma\}=(\bbF_q^{\times})^2}.$ First, we show that the Drinfeld modular forms for such $\Gamma$ are sections of a log canonical bundle on the associated stacky Drinfeld modular curve $\sX_{\Gamma}.$ This solves Gekeler's problem for groups satisfying our hypotheses, assuming one can compute the generators and relations of the log canonical ring of the stacky curve. Recall that under the assumption that $q$ is odd, we know that $k/2$ is an integer when $M_{k,l}(\Gamma)\neq 0,$ i.e.\ when we have non-zero modular forms of weight $k$ and type $l.$ 
		\begin{theorem}[Theorem \ref{thm: forms to differentials} in the text]
			Let $q$ be an odd prime and let $\Gamma\leq \GL_2(\bbF_q[T])$ be a congruence subgroup of $\GL_2(\bbF_q[T])$ such that $\det(\Gamma)=(\bbF_q^{\times})^2.$ 
			Let $\Delta$ be the divisor supported at the cusps of the stacky modular curve $\sX_{\Gamma}$ with the rigid analytic coarse space $X_{\Gamma}^{\text{an}}=\Gamma\setminus(\Omega\cup \bbP^1(\bbF_q(T))).$ 
			There is an isomorphism of graded rings \[M(\Gamma)\cong R(\sX_{\Gamma},\Omega^1_{\sX_{\Gamma}}(2\Delta)),\] where $\Omega^1_{\sX_{\Gamma}}$ is the sheaf of differentials on $\sX_{\Gamma}.$ The isomorphism of algebras is given by the isomorphisms of components $M_{k,l}(\Gamma)\to H^0(\sX_{\Gamma},\Omega^1_{\sX_{\Gamma}}(2\Delta)^{\otimes k/2})$ given by $f\mapsto f(z)(dz)^{\otimes k/2}$ for each $k\geq 2$ an even integer.  
		\end{theorem}
		
		To handle the general case of congruence subgroup $\Gamma$ which may contain matrices with non-square determinants, we consider the normal subgroup $\Gamma_2=\{\gamma\in \Gamma: \det(\gamma)\in (\bbF_q^{\times})^2\}$ of $\Gamma.$ We compare the algebras of Drinfeld modular forms for $\Gamma$ and $\Gamma_2$ and arrive at the following result. Note that this reduces giving an answer to Gekeler for the congruence subgroups $\Gamma$ to computing log canonical rings of stacky Drinfeld modular curves.
		\begin{theorem}[Theorem \ref{thm: decomp of mod forms} in the text]
			Let $q$ be a power of an odd prime. Let $\Gamma\leq \GL_2(\bbF_q[T])$ be a congruence subgroup containing the diagonal matrices in $\GL_2(\bbF_q[T]).$ Let $\Gamma_2=\{\gamma\in \Gamma: \det(\gamma)\in (\bbF_q^{\times})^2\}.$ As rings, we have
			$M(\Gamma)\cong M(\Gamma_2),$
			with \[M_{k,l}(\Gamma_2)=M_{k,l_1}(\Gamma)\oplus M_{k,l_2}(\Gamma)\] on each graded piece, where $l_1,l_2$ are the two solutions to $k\equiv 2l\pmod{q-1}.$ 
		\end{theorem}
		
		Finally, we generalize the previous comparison theorem to a larger class of subgroups $\Gamma'\leq \Gamma,$ where $\Gamma$ is some chosen or distinguished congruence subgroup as above. This idea was proposed in correspondence by Gebhard B\"ockle, as was the proof technique which we execute. This result is similar to classical results about nebentypes of modular forms. 
		\begin{theorem}[Theorem \ref{thm: generalized decomp} in the text]
			Let $q$ be a power of an odd prime and let $\Gamma\leq \GL_2(\bbF_q[T])$ be a congruence subgroup. Let $\Gamma_1=\{\gamma\in \Gamma: \det(\gamma)=1\}.$ Suppose that $\Gamma'$ is a congruence subgroup such that $\Gamma_1\leq \Gamma'\leq \Gamma.$ As algebras
			\[M(\Gamma)=M(\Gamma'),\] and each component $M_{k,l}(\Gamma')$ is some direct sum of components $M_{k,l'}(\Gamma)$ for some nontrivial $l',$ the distinct solutions to $k\equiv [\Gamma:\Gamma']l'\pmod{q-1}.$
		\end{theorem}
		
		\section{Background} 
		
		In the classical number-field setting there is an isomorphism between the ring of modular forms 
		\[M = \bigoplus_{d\geq 0} M_{2d}(\Gamma)\]
		for $\Gamma\leq \SL_2(\bbZ)$ a congruence subgroup, and the ring of global sections of a particular line bundle, such as the sheaf of differentials or the canonical bundle, on the corresponding modular curve. By ring of global sections, we mean a ring of form 
		\[R(\sX,D)=R_D=\bigoplus_{d\geq 0}H^0(\sX,dD),\]
		where $\sX$ is a stacky curve and $D$ is a divisor on $\sX.$
		This allows one to compute algebras of modular forms using the geometry of the moduli space.\\
		
		We will briefly introduce Drinfeld modules, modular forms and modular  curves. In particular we need notation so that we can discuss series of modular forms at cusps of the modular curve, the grading of the algebra of modular forms and some special points on the modular curves. We also mention some of the theory of sections rings for stacks.  
		
		\subsection{Notation and the Drinfeld Setting}
		
		Some references for Drinfeld modular curves are \cite{Gekeler-Curves}, \cite{Gekeler-Invariants} and \cite{Mason-Schweizer-elliptic-pts-Drinfeld-modular-grps}; for Drinfeld modular forms see the survey \cite{Gekeler-survey-Drinfeld-modular-forms} and the papers \cite{Gekeler-jacobians}, \cite{Gekeler-Coeff}, \cite{Breuer-Gekeler-h-function}, \cite{Cornelissen-lvlT} and \cite{Dalal-Kumar-Gamma_0(T)-structure}.\\
		
		Let $\bbF_q$ be the finite field of order $q$ a power of an odd prime. As function-field analogs of $\bbZ,~\bbQ,~\bbR$ and $\bbC$ define the rings $\displaystyle{A=\bbF_q[T], ~K=\operatorname{Frac}(A)=\bbF_q(T), ~K_{\infty}=\bbF_q\left(\!\left(\frac{1}{T}\right)\!\right)},$ the completion of $K$ at the place at $\infty,$ and let $C=\widehat{\overline{K_{\infty}}}$ be the completion of the algebraic closure of $K_{\infty},$ respectively. So, $C$ is an algebraically closed, complete, and non-archimedean field. The Drinfeld-setting version of the upper half-plane $\cH\subset \bbC$ is $\Omega\overset{def}{=}C-K_{\infty}.$ \\
		
		We have the usual discrete valuation $v: K^{\times}\to \bbZ$ given by 
		\[v\left(\frac{\sum_0^n a_iT^i}{\sum_0^m b_iT^i}\right)=m-n,\] where we have assumed $a_n\neq 0$ and $b_m\neq 0,$ and which we extend to the Laurent series $K_{\infty}$ by 
		\[v\left(\sum_{i\geq n}a_iT^i\right)=-n\quad\text{and}\quad v(0)=\infty,\] again with $a_n\neq 0.$ The corresponding metric, which we extend to $C,$ is the non-archimedean norm defined by $|f|_{\infty}=q^{-v(f)}.$ This $|\cdot|$ is the extension of the $\infty$-adic absolute value to $C,$ see e.g.\ \cite[Section $(2.2)$]{Poonen-DrinfeldIntro}. We say $0\neq a\in A$ has $|a|_{\infty}=q^{\deg a}$ and $|0|_{\infty}=0.$ \\
		
		
		Note that the group $\GL_2(A)$ acts on $\Omega$ by M\"obius transformations as $\SL_2$ acts on $\cH,$ but $\det(\gamma)\in \bbF_q^{\times}$ for $\gamma\in \GL_2(A).$ Let $N\in A$ be a non-constant, monic polynomial and let $\Gamma(N)$ be the subgroup of $\GL_2(A)$ consisting of matrices congruent to $\begin{psmallmatrix}1&0\\0&1\end{psmallmatrix}$ modulo $N.$ A subgroup $\Gamma$ of $\GL_2(A)$ such that $\Gamma(N)\subseteq \Gamma$ for some $N$ is a \textbf{congruence subgroup} and we call such an $N$ of the least degree the \textbf{conductor} of $\Gamma.$ \\
		
		Note that if $\Gamma\leq \GL_2(A)$ is some congruence subgroup such that 
		for every $\alpha,\alpha'\in \bbF_q^{\times},$ $\Gamma$ contains the matrices of form $\displaystyle{\begin{psmallmatrix}\alpha&0\\0&\alpha'\end{psmallmatrix}},$ that is, the diagonal matrices in $\GL_2(A),$ then we have
		$\det\Gamma=\{\det(\gamma):\gamma\in \Gamma\}=\bbF_q^{\times}.$ In general $\det\Gamma$ is a subgroup of $\bbF_q^{\times},$ such as for example in the class of congruence subgroups we describe with Theorem \ref{thm: forms to differentials}.\\
		
		Let $(\det\Gamma)^2$ be the set of squares of elements in $\det\Gamma,$ that is, let $(\det\Gamma)^2=\{x^2: x\in \det\Gamma\}.$ Let \[\Gamma_2\overset{def}{=}\{\gamma\in \Gamma:\det(\gamma)\in (\det\Gamma)^2\}=\{\gamma\in \Gamma:\det(\gamma)\in (\bbF_q^{\times})^2\}.\]
		When we distinguish $\Gamma\leq \GL_2(A)$ and its subgroup $\Gamma_2,$ we assume that $\Gamma$ contains diagonal matrices, so $\det\Gamma_2=(\bbF_q^{\times})^2.$\\
		
		The condition that $\Gamma$ has all possible determinants is simply for ease of notation, as it is more pleasant to compute congruences modulo $q-1$ rather than $\#\det\Gamma.$ Our emphasis on the case when $q$ is odd is essential as we make repeated use of the fact that $q-1$ is even. Under this assumption, we also have $k/2$ is an integer when we have non-zero modular forms of weight $k$ and type $l.$ \\ 

		We will make use of a kind of  ``parity'' for congruence subgroups for which we introduce the following terminology:
		\begin{definition}
			We say that a congruence subgroup $\Gamma$ is \textbf{square} if there is some $z\in \Omega$ such that the stabilizer $\Gamma_z=\{\gamma\in \Gamma: \gamma z=z\}$ strictly contains 
			$\displaystyle{\bbF_q^{\times}\cong \left\{\begin{psmallmatrix}\alpha&0\\0&\alpha\end{psmallmatrix}:\alpha\in \bbF_q^{\times}\right\}}$ and every $\gamma\in \Gamma_z\setminus \bbF_q^{\times}$ has a square determinant in $\bbF_q^{\times}.$ Likewise, $\Gamma$ is \textbf{non-square} if it contains a stabilizer $\Gamma_z$ for some $z\in \Omega$ strictly larger than $\bbF_q^{\times}$ and that stabilizer $\Gamma_z$ contains some $\gamma$ with $\det\gamma\in \bbF_q^{\times}\setminus (\bbF_q^{\times})^2.$
		\end{definition}
		
		This idea will help distinguish the geometric invariants involved in the computations at the end of this manuscript into two cases. In our application stabilizers are all $\GL_2(A)$-conjugate subgroups of $\bbF_{q^2}^{\times}$ so that one only needs to check for a single point $z\in \Omega$ with a stabilizer $\Gamma_z\supsetneq \bbF_q^{\times}$ whether $\Gamma_z$ contains some matrix with a non-square determinant.   
		
		\subsection{Drinfeld modules}
		
		The theory of Drinfeld modules is rich in both algebraic and analytic structure. Both interpretations and their equivalence are important in understanding the moduli spaces of Drinfeld modules of a given rank. We state only what we need for our computation of the canonical ring of certain log-stacky moduli spaces and the corresponding algebras of Drinfeld modular forms. 
		
		\subsubsection{Analytic Approach}
		
		We give a quick description of Drinfeld modules as lattice quotients. 
		Following Breuer \cite{Breuer-Gekeler-h-function}, we say an $A$-submodule of $C$ of form $\Lambda = \omega_1A+\cdots+\omega_rA,$ for $\omega_1,\cdots, \omega_r\in C$ some $K_{\infty}$-linearly independent elements, is an \textbf{$A$-lattice of rank $r.$} The \textbf{exponential function of $\Lambda$}, $e_{\Lambda}:C\to C,$ defined by
		\[e_{\Lambda}(z)\overset{def}{=}z\prod_{0\neq \lambda \in \Lambda}\left(1-\frac{z}{\lambda}\right) \] is holomorphic in the rigid analytic sense (see e.g.\ \cite[Definition $2.2.1$]{Frensel-vanderPut-Rigid-Analytic_Geom}), surjective, $\bbF_q$-linear, $\Lambda$-periodic and has simple zeros on $\Lambda.$\\
		
		We characterize the notion of an $\bbF_q$-linear function with the following result.
		\begin{lemma}
			Let $\bbK$ be a field of characteristic $p$ containing $\bbF_q.$ A given $f(X)\in \bbK[X]$ is \textbf{$\bbF_q$-linear} (i.e.\ $f(\alpha X)=\alpha f(X)$ for all $\alpha \in \bbF_q$) if and only if $\displaystyle{f(X)=\sum_{i=0}^n a_iX^{q^i}}$ for some $n\geq 0$ and $a_0,\ldots, a_n\in K.$
		\end{lemma}
		
		Let $C\{X^q\}\overset{def}{=}\{a_0X+a_1X^q+\cdots+a_nX^{q^n}: a_0,\cdots, a_n\in C, ~n\geq 0 \}$ denote the non-commutative polynomial ring of $\bbF_q$-linear polynomials over $C,$ with the operation of multiplication given by composition. With this ring defined, we return to exponential functions of lattices. For each $a\in A$ the exponential satisfies the functional equation 
		\[e_{\Lambda}(az)=\varphi_a^{\Lambda}(e_{\Lambda}(z)),\]
		where $\varphi_a^{\Lambda}(X)\in C\{X^q\}$ is some element of degree $q^{r\deg a}.$
		We say a ring homomorphism $\varphi: A\to C\{X^q\}$ given by \[a\mapsto \varphi_a^{\Lambda}\overset{def}{=}a_0(a)X+\cdots+a_{r\deg a}(a)X^{q^{r\deg a}},\]
		(an $\bbF_q$-algebra monomorphism) is a \textbf{Drinfeld module of rank $r$} if the coefficient with largest index is non-zero.
		
		\subsubsection{Algebraic Approach}
		
		We recall, without any proofs, some facts concerning the algebraic theory which corresponds to the definition above. A more complete discussion of these next facts is found in \cite[Definition $3.1.3$]{Papikian-Drinfeld-modules} and \cite[Lemma $3.1.4$]{Papikian-Drinfeld-modules}. We are mostly interested in the notation.\\ 
		
		We state the following result so that when we define a moduli space of Drinfeld modules, we can make sense of Drinfeld modules over an arbitrary base scheme and therefore eventually have a well-defined category fibered in groupoids when we consider moduli stacks later. 
		
		\begin{theorem}\cite[Page $65$]{Waterhouse-into-aff-grp-schemes}\label{t: endomorphisms of bbG_a}
			Let $B$ be an $A$-algebra, and let $\bbG_{a,B}$ denote the affine additive group scheme over $B$ represented by $\Spec B[t].$ The set $\End_{\bbF_q}(\bbG_{a,B})$ of $\bbF_q$-linear endomorphisms of $\bbG_{a,B},$ is 
			$\End_{\bbF_q}(\bbG_{a,B})\cong B\{X^q\}.$ 
		\end{theorem}
		
		Finally, we can introduce algebraic Drinfeld modules over any scheme.
		\begin{definition}
			A \textbf{Drinfeld module} of rank $r$ over an $A$-scheme $S$ is a pair
			$(E,\varphi)$ consisting of:
			\begin{itemize}
				\item a $\bbG_a$-bundle $E$ (i.e.\ an additive group scheme) over $S$ such that for all $U=\Spec B$ an affine open subset of $S$ for $B$ an $A$-algebra in the Zariski topology on $S,$ there is an isomorphism $\psi: E|_U\overset{\sim}{\to} \bbG_{a,B}$ of group schemes over $U$
				\item a ring homomorphism $\varphi: A\to \End(E)$
			\end{itemize}
			such that for any family of pairs $(U_i,\psi_i)$ which form a trivializing cover of $E$ (i.e.\ $U_i=\Spec B_i$ are an affine open cover and $\psi_i:E_{\pi^{-1}(U_i)}\overset{\sim}{\to}\bbG_{a,B_i}$ are local isomorphisms of additive group schemes), the morphism $\varphi$ restricts to give maps $\varphi_i:A\to \End(\bbG_{a,B_i})$ of the form $\varphi_i(T)=TX+b_{1,i}X^q+\cdots+b_{r,i}X^{q^r},$ compatible with the transition functions $\psi_{ji}=\psi_i\circ \psi_j^{-1},$ i.e.\ $\varphi_j\circ \psi_{ij}=\psi_{ij}\circ \varphi_i$ on all intersections $U_{ij}=U_i\cap U_j.$
		\end{definition}
		
		\begin{remark}
			In the special case when we consider Drinfeld modules over a field, the algebraic definition of a Drinfeld module is simpler. In particular, we have $E=\bbG_a,$ and we do not need any of the trivializations of our bundle as we are working over a single affine scheme. Therefore, it suffices to provide a ring homomorphism $\varphi: A\to \End(\bbG_a).$ We do not make further explicit use of the algebraic definition of Drinfeld modules in this article beyond the following examples.
		\end{remark}
		
		\begin{example}\cite{Carlitz-class-of-poly}\label{ex: Carlitz module}
			The \textbf{Carlitz module} is the rank $1$ Drinfeld module defined by \[\varphi(T)=TX+X^q,\] and corresponds to the lattice $\overline{\pi}A\subset \Omega.$ Here, $\overline{\pi}\in K_{\infty}(\sqrt[q-1]{-T})$ is the \textbf{Carlitz period}, defined up to a $(q-1)$st root of unity. We fix one such $\overline{\pi}$ once and for all.\\
			
			As an algebraic Drinfeld module, the Carlitz module is the ring homomorphism 
			\begin{align*}
				\varphi: &A\to C\{X^q\}\\
				&T\mapsto TX+X^q 
			\end{align*}
			which is a rank $1$ module since $\deg (TX+X^q)=q=|T|_{\infty}^1,$ over the $A$-scheme $\Spec C.$ 
		\end{example}
		
		\begin{example}
			Let $z\in \Omega,$ and consider the rank $2$ lattice $\Lambda_z=\overline{\pi}(zA+A).$ The associated Drinfeld module of rank $2$ is 
			\begin{align*}
				\varphi^z: &~A\to C\{X^q\}\\
				&~T\mapsto TX+g(z)X^q+\Delta(z)X^{q^2},
			\end{align*}
			where $g$ and $\Delta$ are Drinfeld modular forms of type $0$ and weights $q-1$ and $q^2-1$ respectively. We will define Drinfeld modular forms in the next section. This is analogous to defining an elliptic curve by a short Weierstrass equation whose coefficients are values of Eisenstein series.\\
			
			Here we have written an algebraic Drinfeld module of rank $2$ over an affine $A$-scheme similarly to Example \ref{ex: Carlitz module}. The Carlitz period $\overline{\pi}$ serves to normalize the coefficients of the series expansion of $g$ and $\Delta$ at the cusps of $\GL_2(A)$ so that those coefficients are elements of $A.$
		\end{example}
		
		\subsection{Stacks and Section Rings}
		
		See \cite{Alper-Stacks-and-Moduli} for a general stacks reference; see \cite{VZB} for an excellent and comprehensive reference on computing canonical rings of stacky curves; and see \cite{ODorney-canonical-rings-Q-divisors-on-P1}, \cite{Cerchia-Franklin-ODorney-Qdiv-Ell-curves}, and \cite{Landesman-Ruhm-Zhang-Spin-canonical-rings} for useful generalizations of \cite{VZB} that we sometimes use for the Drinfeld setting. We are most interested in Deligne-Mumford stacks for this work, so some facts and examples will be specialized to that case, but we indicate when this occurs. We also discuss rigid analytic stacks and GAGA for rigid analytic and algebraic stacks, but leave that theory for a later section. \\ 
		
		It is shown in e.g.\ \cite[Corollary $1.4.3$]{Laumon-cohomology-Drinfeld-modular-varieties} that the moduli space of rank $r$ Drinfeld modules over the category of schemes of characteristic $p$ is representable by a Deligne-Mumford algebraic stack of finite type over $\bbF_p.$ 
		One is able to compute the graded rings of global sections of line bundles on stacks which represent the Drinfeld moduli problems by means of geometric invariants with results that are slight variants on the theory in \cite{VZB}. We will follow \cite{VZB} in describing this computation, stating only select facts that we will need.\\
		
		Recall from \cite[Sections $5.2$ and $5.3$]{VZB}, a \textbf{stacky curve} $\sX$ over a field $\bbK$ is a smooth, proper, geometrically connected Deligne-Mumford stack of dimension $1$ over $\bbK$ that contains a dense open subscheme. We need no assumptions about this base field to define a stacky curve. Every stacky curve $\sX$ over a field $\bbK$ has a \textbf{coarse space morphism} $\pi:\sX\to X$ with $X$ a smooth scheme over $\bbK$ (called the coarse space), and this map $\pi$ is unique up to unique isomorphism. Indeed, $\pi$ is universal for morphisms from $\sX$ to schemes. What is more, for any algebraically closed field $F$ containing $\bbK,$ the set of isomorphism classes of $F$-points on $\sX$ and the $F$-points of $X$ are in bijection. Note that \'etale locally on the coarse space $X,$ a stacky curve $\sX$ is the quotient of an affine scheme by a finite (constant) group $G\leq \Aut(X).$\\
		
		A \textbf{point} of a stack $\sX$ is a map $\Spec F\to \sX$ for $F$ some field, and to a point $x,$ we associate its stabilizer $G_x\overset{def}{=}\underline{\Isom}(x,x),$ a functor which is a representable by an algebraic space. If $G_x$ is a finite group scheme, say that $\sX$ is \textbf{tame} if $\deg G_x$ is not divisible by $\operatorname{char}(F)$ for any $x\in \sX.$ We say a point $x$ with $G_x\neq \{1\}$ is a \textbf{stacky point}.\\
		
		Continuing the notation of the last paragraphs, let $\pi:\sX\to X$ be a coarse space morphism. A \textbf{Weil divisor} is a finite formal sum of irreducible closed substacks of codimension $1$ over $\bbK.$ On a smooth Deligne-Mumford stack, every Weil divisor is Cartier. Any line bundle $\sL$ on $\sX$ is isomorphic to $\cO_{\sX}(D)$ for some Cartier divisor $D.$ Finally, there is an isomorphism of sheaves on the Zariski site of $X:$
		\[ \cO_X(\lfloor D \rfloor)\overset{\sim}{\to}\pi_*\cO_{\sX}(D), \]
		where \[\lfloor D\rfloor=\Big\lfloor\sum_ia_iP_i\Big\rfloor\overset{def}{=}\sum_i\Big\lfloor\frac{a_i}{\#G_{P_i}}\Big\rfloor\pi(P_i).\]
		
		\begin{example}\label{ex: sheaf of differentials}
			Let $f:\sX\to \sY$ be a morphism of stacky curves with coarse spaces $X$ and $Y=\Spec k$ for $k$ some field respectively. The \textbf{sheaf of differentials} $\Omega^1_{\sX}=\Omega^1_{\sX/\Spec k}$ is the sheafification (see \cite[Section $2.2.3$]{Alper-Stacks-and-Moduli} for sheafification) of the presheaf on $\sX_{\text{\'et}},$ the small \'etale site on $\sX$ (i.e.\ the category of schemes which are \'etale over $\sX$) given by 
			\[ (U\to \sX)\mapsto \Omega^1_{\cO_{\sX}(U)/f^{-1}\cO_{\sY}(U)},\]
			where $\cO_{\sX}$ and $\cO_{\sY}$ denote the structure sheaves on $\sX$ and $\sY$ respectively. See e.g.\ \cite[Example $4.1.2$]{Alper-Stacks-and-Moduli} for more details on structure sheaves for Deligne-Mumford stacks. 
		\end{example}
		
		Every smooth, projective curve $X$ may be treated as a stacky curve with nothing stacky about it. On the other hand the stack quotient $[X/G]$ for a finite group $G\leq \Aut(X)$ is a stacky curve. We know from e.g.\ \cite[Remark $5.2.8$]{VZB} that Zariski locally, every stacky curve is the quotient of a smooth, affine curve by a finite group, so locally, stacky curves have a quotient description $[X/G]$ as above. Recall from \cite[Lemma $5.3.10.(b)$]{VZB} that the stabilizer groups of a tame stacky curve are isomorphic to the group of roots of unity $\mu_n$ for some $n.$ In order to discuss Drinfeld moduli stacks, we introduce two more stacky notions.\\
		
		Let us consider $\sX$ a stacky curve as above and $x:\Spec \bbK\to\sX$ a point on $\sX$ with stabilizer $G_x.$ A \textbf{residue gerbe} at $x$ is the unique monomorphism $G_x\hookrightarrow \sX$ through which $x$ factors. As in \cite{VZB} we treat residue gerbes as fractional points on a stacky curve. We will say a \textbf{gerbe} over the stacky curve $\sX$ is a smooth, proper, geometrically connected Deligne-Mumford stack of dimension $1$ where every point has a stabilizer containing some nontrivial group. Note that a gerbe is almost a stacky curve, except that it does not contain a dense open subscheme and indeed, each point is fractional in the sense above. Let $\sX$ denote a geometrically integral Deligne-Mumford stack of relative dimension $1$ over a base scheme $S$ whose generic point has stabilizer $\mu_n$ for some $n.$ There exists a stack, denoted $\sX/\!/\mu_n,$ called the \textbf{rigidification} of $\sX,$ and a factorization 
		\[\sX\overset{\pi}{\to}\sX/\!/\mu_n\to S\] 
		such that $\pi$ is a $\mu_n$-gerbe (i.e.\ for each point $x$ of $\sX,$ the stabilizer $G_x$ contains $\mu_n$) and the stabilizer of any point in $\sX/\!/\mu_n$ is the quotient of the stabilizer of the corresponding point in $\sX$ by $\mu_n.$
		
		\begin{remark}
			In the factorization above, since $\pi$ is a gerbe and furthermore is \'etale, the sheaf of relative differentials $\sX\to \sX/\!/\mu_n$ is $0,$ i.e.\ the gerbe does not affect sections of relative differentials (over the base scheme), nor the canonical ring which we define for stacks below. In particular, we can identify canonical divisors $K_{\sX}\sim \pi^*K_{\sX/\!/\mu_n},$ and the corresponding canonical rings are isomorphic. 
		\end{remark}
		
		In particular, we treat seriously the stackiness of the moduli when we compute the following homogeneous coordinate rings on Drinfeld modular curves.  
		\begin{definition}
			Let $\sX$ be a stacky curve over a field $k$ and let $\sL$ be an inveritble sheaf on $\sX.$ The \textbf{section ring} of $\sL$ on $\sX$ is the ring 
			\[R(\sX,\sL)=\bigoplus_{d\geq 0}H^0(\sX,\sL^{\otimes d}).\]
			If $\sL\cong \cO_{\sX}(D)$ for some Weil divisor $D$ in particular, we can equivalently write 
			\[R_D=\bigoplus_{d\geq 0}H^0(\sX,dD).
			\]
		\end{definition}		
		
		Finally, for readability of our main results, we introduce some terminology inspired by \cite[Definition $5.6.2$]{VZB} and \cite[Proposition $5.5.6$]{VZB}. Let $\sX$ be a tame stacky curve over an algebraically closed field $\bbK$ with coarse space $X.$ A Weil divisor $\Delta$ on $\sX$ is a \textbf{log divisor} if $\Delta=\sum_i P_i$ is an effective divisor given as a sum of distinct points (stacky or otherwise) on $\sX.$ We have slightly generalized Voight and Zureick-Brown's notion of log divisor for a reason which we discuss in Sections \ref{s: Drinfeld modular curves} and \ref{s: Proof of main thm}. This generalization only means more involved calculations, and at worst less aesthetic results, when computing log canonical rings, which is why in \cite[Remark $5.4.6$]{VZB} the authors restrict their log divisors more than in this article.\\ 
		
		By \cite[Proposition $5.5.6$]{VZB} if $K_{\sX}$ and $K_X$ are canonical divisors on a stacky curve $\sX$ and its coarse space $X$ respectively, then there is a linear equivalence
		\[K_{\sX}\sim K_X+R=K_X+\sum_x \left(\deg G_x-1\right)x,\]
		where $G_x$ is the stabilizer of a closed substack $x\in \sX,$ and the sum above is taken over closed substacks of $\sX$ of codimension $1.$ Finally, a \textbf{log canonical divisor} on a stacky curve $\sX$ has form $K_{\sX}+\Delta,$ where $K_{\sX}$ is a canonical divisor, and $\Delta$ is a log divisor on $\sX.$

		\section{Drinfeld Modular Forms}
		
		In this section we introduce Drinfeld modular forms. The technical conditions of the rigid analytic space in which we work makes it necessary to introduce some facts about the projective line $\bbP^1(C)$ before we begin in earnest on a study of modular forms. We discuss rigid analytic spaces in more detail in the following sections.
		
		\begin{definition}\label{d: parameter at infty}
			Let $\overline{\pi}\in K_{\infty}(\sqrt[q-1]{-T})$ be a fixed choice of the Carlitz period (recall Example \ref{ex: Carlitz module}). We define a \textbf{parameter at infinity} 
			\[u(z)\overset{def}{=}\frac{1}{e_{\overline{\pi}A}(\bar{\pi}z)}=\frac{1}{\bar{\pi}e_A(z)}=\bar{\pi}^{-1}\sum_{a\in A}\frac{1}{z+a}.\]
		\end{definition}
		\begin{remark}
			In the Drinfeld setting, $\overline{\pi}$ plays the role of the constant $2\pi i\in \bbC$ in the parameter $q=e^{2\pi i z}$ at infinity from the classical setting. That is, it is a normalization  factor so that the series expansion coefficients for certain generating modular forms at cusps are elements of $A.$ 
		\end{remark}
		
		One fact about this parameter which we will use later in our consideration of modular forms is the following. 
		\begin{lemma}\cite[Page $494$]{Gekeler-survey-Drinfeld-modular-forms}\label{l: u(a/d)=d/au}
			For each $\alpha\in \bbF_q^{\times}$ we have $\displaystyle{u\left(\alpha z\right)=\alpha^{-1}u(z)}.$
		\end{lemma}	
		\begin{proof}
			Since the exponential function $e_A$ is an $\bbF_q$-linear power series we know $e_A(\alpha z)=\alpha e_A(z),$ so $\displaystyle{u(\alpha z)=\frac{1}{\bar{\pi}e_A(\alpha z)}=\frac{1}{\bar{\pi}\alpha e_A(z)}=\alpha^{-1}u(z)}.$
			
			
		\end{proof}
		
		Now we are able to define a fundamental object of study for this article. 
		
		\begin{definition}\cite[Definition $(3.1)$]{Gekeler-Curves}
		\label{def: Drinfeld modular form}
			Let $\Gamma\leq \GL_2(A)$ be a congruence subgroup. A \textbf{modular form} of \textbf{weight} $k\in \bbZ_{\geq0}$ and \textbf{type} $l\in \bbZ/((q-1) \bbZ)$ for $\Gamma$ is a holomorphic function $f:\Omega\to C$ such that 
			\begin{enumerate}
				\item $f(\gamma z)=\det(\gamma)^{-l}(cz+d)^kf(z)$ for all $\displaystyle{\gamma=\begin{psmallmatrix}a&b\\c&b\end{psmallmatrix}\in \Gamma},$ and
				\item $f$ is holomorphic at the cusps of $\Gamma.$
			\end{enumerate}
			If such an $f$ satisfies only condition $(1),$ we say $f$ is simply \textbf{weakly modular} (of weight $k,$ type $l,$ for $\Gamma$).
		\end{definition}
		
		\begin{remark}
			There are several interpretations of the second condition when $\Gamma$ has a single cusp, or generally about the condition of holomorphy at the cusp $\infty$:
			\begin{enumerate}
				\item \cite[$(2.2.\mathrm{iii})$]{Gekeler-Invariants} The condition is equivalent to $f$ being bounded on $\{z\in \Omega:|z|_{\infty}\geq 1\},$ where $|\cdot|_{\infty}$ is the $\infty$-adic absolute value; 
				\item \cite[Definition $3.5.(\mathrm{iii})$]{Gekeler-survey-Drinfeld-modular-forms} $f$ has a series expansion at cusps: 
				\[f(z)=\sum_{n\in \bbZ}a_nu(z)^n, ~a_n\in C,\]
				where $u$ is the parameter at $\infty,$ with a positive radius of convergence. The second condition means that $a_n=0$ for all $n<0.$
			\end{enumerate}
		\end{remark}
		
		\begin{remark}
			The observation from \cite[Definition $(5.7)$]{Gekeler-Coeff} that if $f$ is Drinfeld modular form, then $f(z+b)=f(z)$ for any $b\in A$ means that although not literally a Fourier series, the series expansion of a modular form at the cusps of some congruence subgroup is the Drinfeld setting equivalent to a Fourier series.   
		\end{remark}
		
		We introduce some terminology and notation respectively in the next definition.
		
		\begin{definition}
			Write $M_{k,l}(\Gamma)$ for the finite-dimensional $C$-vector space of Drinfeld modular forms for $\Gamma\leq \GL_2(A)$ with weight $k$ and type $l.$ The graded ring $M(\Gamma)$ of modular forms is 
			\[M(\Gamma)=\bigoplus_{\substack{k\geq 0\\l\Mod{q-1}}} M_{k,l}(\Gamma)\]
			since $M_{k,l}\cdot M_{k',l'}\subset M_{k+k',l+l'}.$
		\end{definition}
		
		Now we can introduce some non-trivial facts about Drinfeld modular forms.
		
		\begin{lemma}\cite[Remark $5.8.\mathrm{iii}$]{Gekeler-Coeff}\label{l: u-series coeffs determine form}
			If $f(z)\in M_{k,l}(\Gamma)$ has a $u$-series expansion $f(z)=\sum_{n\geq 0}a_nu^n,$ then the coefficients $a_i$ uniquely determine $f.$
		\end{lemma}
		
		The weight and type of Drinfeld modular forms are not independent. 
		
		\begin{lemma}\cite[Remark $(5.8.\mathrm{i})$]{Gekeler-Coeff}\label{l: weight-type}
			If $M_{k,l}(\Gamma)\neq 0,$ then $k\equiv 2l\pmod{q-1}.$
		\end{lemma}
		\begin{proof}
			See \cite[Remark $(5.8.\mathrm{iii})$]{Gekeler-Coeff}
		\end{proof}
		
		\begin{example}
			Some famous Drinfeld modular forms are the $\GL_2(A)$-forms: $g$ of weight $q-1$ and type $0, ~\Delta$ of weight $q^2-1$ and type $0,$ and $h$ of weight $q+1$ and type $1.$ We know from Goss and Gekeler respectively, see for example \cite[Theorem $(3.12)$]{Gekeler-survey-Drinfeld-modular-forms}, that 
			\[\bigoplus_{k\geq 0} M_{k,0}(\GL_2(A))=C[g,\Delta] \quad \text{ and }\quad \bigoplus_{\substack{k\geq 0\\l\Mod{q-1}}} M_{k,l}(\GL_2(A))=C[g,h].\]
		\end{example}
		
		\begin{example}\cite[Section $8$]{Gekeler-Coeff}
		\label{example: Eisenstein series for Gamma0(T)}
			The function \[E(z)\overset{def}{=}\overline{\pi}^{-1}\sum\limits_{\substack{a\in A\\\text{monic}}}\left(\sum_{b\in A}\frac{a}{az+b}\right)\]
			is an analog to an Eisenstein series of weight $2$ over $\bbQ,$ and we can define a Drinfeld modular form 
			\[E_T(z)\overset{def}{=}E(z)-TE(Tz)\]
			of weight $2$ and type $1$ for $\Gamma_0(T),$ the congruence subgroup of $\GL_2(A)$ containing matrices $\displaystyle{\begin{psmallmatrix}a&b\\c&d\end{psmallmatrix}}$ with $c\equiv 0\Mod T.$
		\end{example}
		
		\section{Drinfeld Modular Curves}
		\label{s: Drinfeld modular curves}
		
		%
		Let us consider the moduli space of rank $2$ Drinfeld modules, first as a rigid analytic space, then as moduli schemes, and finally as log stacky curves. We recall some definitions we need to discuss rigid anaytic spaces, which are a natural means to discuss the affine Drinfeld modular curves as quotients of the Drinfeld ``upper half-plane'' $\Omega$ by congruence subgroups. A more thorough treatment and reference for rigid analytic geometry is \cite{Frensel-vanderPut-Rigid-Analytic_Geom}. We will specialize to rigid analytic spaces over $C$ for readability.\\
		
		We need the following intermediate definitions to define a rigid analytic space. 
		\begin{definition}
			Let $z_1,\cdots, z_n$ denote some variables and let $\bbK$ denote a non-archemedean, valued field. Let $T_n=\bbK\langle z_1,\cdots, z_n\rangle$ be the $n$-dimensional $\bbK$-algebra which is the subring of the ring of formal power series $\bbK[\![z_1,\cdots, z_n]\!]$ 
			\[T_n=\left\{\sum_{\alpha} c_{\alpha}z_1^{\alpha_1}\cdots z_n^{\alpha_n}\in \bbK[\![z_1,\cdots, z_n]\!] :\lim c_{\alpha}=0\right\},\] 
			where $\alpha = (\alpha_1,\cdots,\alpha_n).$ An \textbf{affinoid algebra} $A$ over $\bbK$ is a $\bbK$-algebra which is a finite extension of $T_n$ for some $n\geq 0.$
		\end{definition} 
		
		Associated to an affinoid algebra $A$ over a field $\bbK$ is a corresponding \textbf{affinoid space} $\operatorname{Sp}(A),$ the set of its maximal ideals. 
		
		\begin{definition}\cite[Definition $2.4.1$]{Frensel-vanderPut-Rigid-Analytic_Geom}
			Let $X$ be a set. A \textbf{G-topology} on $X$ consists of the data:
			\begin{enumerate}
				\item A family $\sF$ of subsets of $X$ such that $\emptyset, X\in \sF$ and if $U,V\in \sF,$ then $U\cap V\in \sF,$ and 
				\item For each $U\in \sF$ a set $\operatorname{Cov}(U)$ of coverings of $U$ by elements of $\sF$
			\end{enumerate}
			such that the following conditions are met:
			\begin{itemize}
				\item $\{U\}\in \operatorname{Cov}(U)$
				\item For each $\cU,V\in \sF$ with $V\subset \cU$ and $\cU\in \operatorname{Cov}(U),$ the covering $\displaystyle{\cU\cap V\overset{def}{=}\{U'\cap V: U'\in \cU\}}$ belongs to $\operatorname{Cov}(V)$
				\item Let $U\in \sF,$ let $\{U_i\}_{i\in I}\in \operatorname{Cov}(U)$ and let $\cU_i\in \operatorname{Cov}(U_i).$ The union \[\bigcup_{i\in I} \cU_i\overset{def}{=}\{U':U'\text{ belongs to some }\cU_i \}\] is an element of $\operatorname{Cov}(U).$
			\end{itemize}
			We say the $U\in \sF$ are \textbf{admissible sets} and the elements of $\operatorname{Cov}(U)$ are \textbf{admissible coverings}.
		\end{definition}
		
		Finally, we come to the point:
		\begin{definition}\cite[Definition $4.3.1$]{Frensel-vanderPut-Rigid-Analytic_Geom}
			A \textbf{rigid analytic space} is a triple $(X,T_X,\cO_X)$ consisting of a set $X,$ a $G$-topology $T_X$ on $X$ and a structure sheaf of $C$-algebras $\cO_X$ on $X$ for which there exists an admissible open covering $\{X_i\}$ of $X$ such that each $(X_i,T_{X_i},\cO_{X_i})$ is an affinoid over $C$ and $U\subset X$ belongs to $T_X$ if and only if $U\cap X_i$ belongs to $T_X$ for each $i.$ 
		\end{definition}
		
		For the well-definedness of Drinfeld modular curves, we recall some analytic properties of $\Omega.$
		Since $\Omega=\bbP^1(C)-\bbP^1(K_{\infty}),$ and $\bbP^1(K_{\infty})$ is compact in the rigid analytic topology, we know from \cite[Section $1.2$]{Gekeler-jacobians} that $\Omega$ is a rigid analytic space. The action by  a congruence subgroup $\Gamma\leq \GL_2(A)$ on $\Omega$ by M\"obius transformations has finite stabilizer for each $z\in \Omega,$ and as in \cite[Section  $(2.2)$]{Gekeler-jacobians}, $\Gamma\setminus \Omega$ is a rigid analytic space.\\
		
		Recall that for any scheme $S$ of locally finite type over a complete, non-archimedean field of finite characteristic $p,$ there is a rigid analytic space $S^{\text{an}}$ whose points coincide with those of $S$ as sets. In fact, there is an \textbf{analytification functor} from the category of schemes over $C$ to the category of rigid analytic spaces, so if $X$ is a smooth algebraic curve over $C,$ then there is a rigid analytic space $X^{\text{an}}$ whose points are in bijection with the $C$-points of $X.$ 
		
		\begin{theorem}\cite{Drinfeld-elliptic-modules}
			There exists a smooth, irreducible, affine algebraic curve $Y_{\Gamma}$ over $C$ such that $\Gamma\setminus \Omega$ and the underlying (rigid) analytic space $Y_{\Gamma}^{\text{an}}$ of $Y_{\Gamma}$ are canonically isomorphic as rigid analytic spaces over $C.$ 
		\end{theorem}
		\begin{remark}
			This underlying rigid analytic space is the \textbf{analytification} (see \cite[Example $4.3.3$]{Frensel-vanderPut-Rigid-Analytic_Geom}) of $Y_{\Gamma}.$ 
		\end{remark}
		
		\begin{definition}
			We call the affine curves $Y_{\Gamma}$ whose analytification $Y_{\Gamma}^{\text{an}}$ are isomorphic to $\Gamma\setminus \Omega$ as rigid analytic spaces over $C$ affine \textbf{Drinfeld modular curves.} Since $Y_{\Gamma}$ is smooth and affine, it admits a smooth projective model which $X_{\Gamma}$ which is the projective Drinfeld modular curve. 
		\end{definition}
		
		
		\begin{remark}
			In the spirit of \cite[Section $6.2$]{VZB}, we say a projective Drinfeld modular curve $X_{\Gamma}$ is the \textbf{algebraization} of some rigid analytic space $\Gamma\setminus (\Omega\cup \bbP^1(K))=X_{\Gamma}^{\text{an}},$ whose points are in bijection with the $C$-points of the projective Drinfeld modular curve $X_{\Gamma}.$ 
		\end{remark}
		
		Let $X_{\Gamma}^{\text{an}}\overset{def}{=}\Gamma\setminus(\Omega\cup \bbP^1(K))$ denote a rigid analytic, projective Drinfeld modular curve for some congruence subgroup $\Gamma\leq \GL_2(A).$
		Let $X_{\Gamma}=(X_{\Gamma}^{\text{an}})^{\text{alg}}$ denote the corresponding algebraic Drinfeld modular curve whose $C$-points are in bijection with $X_{\Gamma}^{\text{an}}.$ This modular curve is not a stacky curve since there is a uniform $\mu_{q-1}$ stabilizer which we know from the moduli interpretation - each point is fixed by $\displaystyle{Z(GL_2(A))=\{\begin{psmallmatrix}\alpha&0\\0&\alpha\end{psmallmatrix}:\alpha\in \bbF_q^{\times}\}\cong \bbF_q^{\times}}.$ However, as a scheme, $X_{\Gamma}$ is the coarse space of a stacky curve $\sX_{\Gamma}$ given by the stack quotient $[X_{\Gamma}/Z(\GL_2(A))].$ Furthermore, if $\overline{\cM^2}_{\Gamma}$ denotes (Laumon's) Deligne-Mumford stack representing the corresponding moduli problem, then every point of $\overline{\cM^2_{\Gamma}}$ has a stabilizer containing (at least) $\bbF_q^{\times}.$ This $\overline{\cM^2_{\Gamma}}$ is a $\mu_{q-1}$-gerbe over $\sX_{\Gamma},$ i.e.\ $\displaystyle{\sX_{\Gamma}=\overline{\cM^2_{\Gamma}}/\!/\mu_{q-1}}$ is a rigidification of $\overline{\cM^2_{\Gamma}}$:
		\[\overline{\cM^2_{\Gamma}}\to \sX_{\Gamma}\to X_{\Gamma}.\]
		When we discuss \textbf{stacky Drinfeld modular curves} we mean a stacky curve $\sX_{\Gamma}$ as in this paragraph, that is the rigidification of some moduli problem (i.e.\ of one of Laumon's gerbes).\\ 
		
		Next we consider some special points on Drinfeld modular curves. 
		\begin{definition}
			Let $\Gamma\leq \GL_2(A)$ be a congruence subgroup, let $Y_{\Gamma}^{\text{an}}=\Gamma\setminus \Omega$ and let $X_{\Gamma}^{\text{an}}=\Gamma\setminus(\Omega\cup \bbP^1(K)).$ A \textbf{cusp of $X_{\Gamma}^{\text{an}}$} is a point of $X_{\Gamma}^{\text{an}}-Y_{\Gamma}^{\text{an}}.$ 
		\end{definition}
		
		\begin{remark}
			As sets, $X_{\Gamma}^{\text{an}}=\Gamma\setminus(\Omega\cup \bbP^1(K)),$ so since $\GL_2(A)$ acts transitively on $\bbP^1(K)$ we have
			\[\cC_{\Gamma}\overset{def}{=}\{\text{cusps of }X_{\Gamma}^{\text{an}}\}\overset{def}{=}\Gamma\setminus \bbP^1(K)=\Gamma\setminus \GL_2(A)/\GL_2(A)_{\infty},\]
			where $\displaystyle{\GL_2(A)_{\infty}=\{\gamma\in \GL_2(A):\gamma(\infty)=\infty\}=\left\{\begin{psmallmatrix}*&*\\0&*\end{psmallmatrix}\right\}}.$
		\end{remark}
		
		\begin{definition}\label{d: elliptic pt}
			$~$
			\begin{enumerate}
				\item	If $e\in \Omega$ has $(\GL_2(A))_e=\{\gamma\in \GL_2(A):\gamma(e)=e\}$ strictly larger than $\bbF_q^{\times}\cong \left\{\begin{psmallmatrix}\alpha&0\\0&\alpha\end{psmallmatrix}\right\}$ then $e$ is an \textbf{elliptic point on $\Omega$}. In this case, $\GL_2(A)_e\cong \bbF_{q^2}^{\times}.$\\
				
				\item	Let $\Gamma\leq \GL_2(A)$ be a congruence subgroup. A point $e\in \Omega$ is an \textbf{elliptic point for $\Gamma$} if the stabilizer $\Gamma_e$ is strictly larger than $\displaystyle{\bbF_q^{\times}\cong\left\{\begin{psmallmatrix}\alpha&0\\0&\alpha\end{psmallmatrix}:\alpha\in \bbF_q^{\times}\right\}}$ (the center of $\GL_2(\bbF_q)$). 
			\end{enumerate}
		\end{definition}
		
		\begin{remark}\label{remark: unique elliptic point for Omega}
			An elliptic point $e$ on $\Omega$ is a point which is $\GL_2(A)$-conjugate to some element of $\bbF_{q^2}\setminus\bbF_q\hookrightarrow \Omega.$ Fix once and for all an elliptic point $e\in \bbF_{q^2}\setminus\bbF_q$ on $\Omega.$ 
			We write
			\[\operatorname{Ell}(\Gamma)\overset{def}{=}\{\text{elliptic points of }X_{\Gamma}^{\text{an}}\}.\]
		\end{remark}
		
		\begin{remark}\label{remark: tameness of Drinfeld modular curves}
			Note that for $\Gamma\leq \GL_2(A)$ any congruence subgroup, the Drinfeld modular curves $\sX_{\Gamma},$ are \textbf{tame} over $C$ in the sense of \cite[Example $5.2.7$]{VZB}. Indeed, we may describe $\sX_{\Gamma}$ by the stack quotient $[X_{\Gamma}/Z(\GL_2(A))],$ and since $\gcd(\operatorname{char}(C),\#Z(\GL_2(A)))=1$ the quotient is tame. 
		\end{remark}
		
		Recall that $\bbP_{\bbK}^1(a_0,\ldots, a_n)$ denotes the \textbf{weighted projective line} over a field $\bbK$ defined by $\displaystyle{\bbP^1(a_1,\ldots, a_n)=\Proj(\bbK[x_0,\ldots, x_n])},$ where each $x_i$ is an indeterminant of degree $a_i.$
		
		\begin{example}[The $j$-line]
			Let $X(1)=\GL_2(A)\setminus(\Omega\cup\bbP^1(K))$ be the ``usual'' $j$-line. Let $\overline{\cM^2_{\Gamma}}$ be the Deligne-Mumford stack representing the corresponding moduli problem (including cusps). The stack $\overline{\cM^2_{\Gamma}}$ is a $\mu_{q-1}$ gerbe over $\sX(1)=[X(1)/Z(\GL_2(A))].$ In other words, $\sX(1)$ is a rigidification $\overline{\cM^2_{\Gamma}}/\!/\mu_{q-1}$: 
			\[\begin{array}{cccl}
				\overline{\cM^2_{\Gamma}}&\overset{\pi}{\to}&\sX(1)&\to X(1)\\
				\bbP^1((q-1)^2,q^2-1)&\overset{\pi}{\to}&\bbP^1(q-1,q+1)&\to\bbP^1(C).
			\end{array}\]
		\end{example}
		
		
		\section{Rigid Stacky GAGA}
		
		We need a precise notion of a rigid analytic stack for rigid stacky GAGA. Since algebraic (stacky) Drinfeld modular curves are Deligne-Mumford stacks, we will specialize the notion of rigid analytic Artin stacks from \cite[Section $5.1.7$]{Emerton-Gee-Hellman-categorical-p-adic-langlands} to Deligne-Mumford rigid analytic stacks.\\
		
		Let $\operatorname{Rig}_C$ denote the category of rigid analytic spaces over $C.$ Equip $\operatorname{Rig}_C$ with the \textbf{Tate-fpqc} topology (see \cite[$2.1$]{Conrad-Temkin-nonarhimedean-analytification-alg-spaces}). The covers in this topology are generated by the admissible Tate coverings (see \cite[Section $4.2$]{Frensel-vanderPut-Rigid-Analytic_Geom}) and the morphisms $\operatorname{Sp}(A)\to \operatorname{Sp}(B)$ for faithfully flat morphisms of affinoid algebras $B\to A.$ By \cite[Theorem $4.2.8$]{Conrad-ampleness-rigid-geom} all representable functors in this topology are sheaves and coherent sheaves satisfy descent. With this site specified, we can define rigid analytic stacks.
		
		\begin{definition}
			A \textbf{stack on $\operatorname{Rig}_C$} is a category fibered in groupoids which satisfies descent for the Tate-fpqc topology. 
		\end{definition}
		
		Next we cite \cite{Emerton-Gee-Hellman-categorical-p-adic-langlands} to define a rigid analytic Artin stack. Good references on Artin stacks, which appear first in \cite{Artin-versal-deformations-algebraic-stacks}, are  \cite{Abramovich-Olsson-Vistoli-tame-stacks-pos-characteristic} and \cite{Abramovich-Olsson-Vistoli-twisted-stable-maps-tame-Artin-stacks}.
		
		\begin{definition}\cite[$5.1.10$]{Emerton-Gee-Hellman-categorical-p-adic-langlands}
			A \textbf{rigid analytic Artin stack} is a stack $\sX$ on $\operatorname{Rig}_C$ such that the diagonal $\Delta_{\sX}:\sX\to \sX\times_C \sX$ is representable by a rigid analytic space, and there exists some rigid analytic space $U$ and a smooth surjective map $U\to \sX.$
		\end{definition}
		
		Now we are equipped to define the version of rigid analytic stack we will consider in application of a rigid GAGA theorem on stacks. 
		
		\begin{definition}
			\label{def: rigid analytic DM stack v2}
			A \textbf{rigid analytic Deligne-Mumford stack} is a rigid analytic Artin stack $\sX$ such that the diagonal $\Delta_{\sX}:\sX\to \sX\times_C\sX$ is representable by a rigid analytic space, quasi-compact and separated for the Tate-fpqc topology. 
		\end{definition}
		
		We will use the following rigid stacky GAGA theorem once we have introduced the necessary terminology. For $X$ a stack, we write $\operatorname{Coh}^{\heartsuit}(X)$ for the full subcategory of (the category of) coherent sheaves on $X$ spanned by objects cohomologically concentrated in degree $0.$ That is, coherent sheaves all of whose non-trivial cohomology groups are only in the degree $0$ position.
		
		\begin{theorem}\cite[Theorem $7.4$]{Porta-Yu-Higher-analytic-stacks-GAGA}\label{t: rigid stacky GAGA}
			Let $A$ be a $\bbK$-affinoid algebra, for $\bbK$ some non-achimedean field. Let $\sX$ be a proper algebraic stack over $\Spec(A).$ The analytification functor on coherent sheaves induces an equivalence of $1$-categories
			\[\operatorname{Coh}^{\heartsuit}(\sX)\overset{\cong}{\to} \operatorname{Coh}^{\heartsuit}(\sX^{\text{an}}).\] 
		\end{theorem}
		
		\section{Proof of the Main Theorems}
		\label{s: Proof of main thm}
		
		We prove our first main result:
		
		\begin{theorem}\label{thm: forms to differentials}
			Let $q$ be an odd prime and let $\Gamma\leq \GL_2(A)$ be a congruence subgroup of $\GL_2(A)$ such that $\det(\Gamma)= (\bbF_q^{\times})^2.$ Let $\Delta$ be the divisor of cusps of the modular curve $\sX_{\Gamma}$ with the rigid analytic coarse space $X_{\Gamma}^{\text{an}}=\Gamma\setminus(\Omega\cup \bbP^1(K)).$ 
			There is an isomorphism of graded rings $M(\Gamma)\cong R(\sX_{\Gamma},\Omega^1_{\sX_{\Gamma}}(2\Delta)),$ where $\Omega^1_{\sX_{\Gamma}}$ is the sheaf of differentials on $\sX_{\Gamma}.$ The isomorphism of algebras is given by the isomorphisms of components $M_{k,l}(\Gamma)\to H^0(\sX_{\Gamma},\Omega^1_{\sX_{\Gamma}}(2\Delta)^{\otimes k/2})$ given by $f\mapsto f(z)(dz)^{\otimes k/2}$ for each $k\geq 2$ an even integer. 
		\end{theorem}
		\begin{proof}
			Suppose $f\in M_{k,l}(\Gamma)$ is some non-zero form. For any $\displaystyle{\gamma=\begin{psmallmatrix}a&b\\c&d\end{psmallmatrix}\in \Gamma}$ we have \[
			f(\gamma z)d(\gamma z)^{\otimes k/2}= (cz+d)^k(\det\gamma)^{-l}\frac{\det\gamma^{k/2}}{(cz+d)^k} f(z)dz^{\otimes k/2},\] 
			where $k\equiv 2l\pmod{\frac{q-1}{2}}.$ All of the factors of automorphy cancel and \[f(\gamma z)d(\gamma z)^{\otimes k/2}=f(z)dz^{\otimes k/2},\] so the differential form $f(z)(dz)^{\otimes k/2}\in H^0(\Omega,\Omega_{\Omega}^{\otimes k/2})$ on the upper half-plane $\Omega$ is $\Gamma$-invariant. As in \cite[Section $(2.10)$]{Gekeler-jacobians}, we know $f(z)(dz)^{\otimes k/2}$ is holomorphic on $\Gamma\setminus \Omega.$
			Since $\displaystyle{\frac{de_A(z)}{dz}=1},$ we have 
			$\displaystyle{\frac{du}{u^2} = -\overline{\pi}dz,}$ so the differential $dz$ in this case has a double pole at $\infty.$
			Since $f$ is holomorphic at the cusps of $\Gamma,$ 
			\[\operatorname{div}(f(z)(dz)^{\otimes k/2}) +k\Delta\geq 0,\]
			and therefore $f(z)(dz)^{\otimes k/2}$ is a global section of the twist by $2\Delta$ of sheaf of holomorphic differentials on the rigid analytic space $X_{\Gamma}^{\text{an}}=\Gamma\setminus(\Omega\cup\bbP^1(K)).$
			We claim this is a global section of (a twist by $2\Delta$ of) the sheaf of differentials on the algebraic stack $\sX.$\\
			
			By rigid analytic GAGA, \cite[Theorem $4.10.5$]{Frensel-vanderPut-Rigid-Analytic_Geom}, we know that the categories of coherent sheaves on the rigid space $\bbP^{n,\text{an}}_C$ and coherent sheaves on $\bbP^n_C$ are equivalent for $n\geq 1$ any integer. Furthermore, every closed analytic subspace of $\bbP^{n,\text{an}}_C$ is the analytification of some closed subspace of $\bbP^n_C.$ So, the sheaf $\Omega_{X_{\Gamma}^{\text{an}}}^1(2\Delta)$ corresponds to the sheaf $\Omega^1_{X_{\Gamma}}(2\Delta)$ on the algebraic curve $X_{\Gamma}$ which is the coarse space of $\sX.$ Finally, by 
			Theorem \ref{t: rigid stacky GAGA}, we know the sheaves $\Omega_{\sX_{\Gamma}^{\text{an}}}^1(2\Delta)$ and $\Omega^1_{\sX_{\Gamma}}(2\Delta)$ on the rigid analytic stacky curve and algebraic stacky curves $\sX_{\Gamma}^{\text{an}}$ and $\sX_{\Gamma}$ respectively are equivalent.\\ 
			
			We have shown that given a modular form of weight $k$ and type $l$ for $\Gamma,$ the differential form $f(z)(dz)^{\otimes k/2}$ on the stacky curve $\sX_{\Gamma}$ is $\Gamma$-invariant and holomorphic at cusps, so therefore is a global section of the sheaf of differentials on the stacky curve. It is well-known that the only such $\Gamma$-invariant differentials are in one-to-one correspondence with modular forms, or, one might observe that the kernel of our homomorphism of algebras is trivial, which completes the typical argument (as in e.g.\ \cite[Chapter $6.2$]{VZB}) for an isomorphism between an algebra of modular forms and a ring of global sections of some line bundle on a stacky curve.
		\end{proof}
		
		Next, we recall our second main result:
		\begin{theorem}
		\label{thm: decomp of mod forms}
			Let $q$ be a power of an odd prime. Let $\Gamma\leq \GL_2(A)$ be a congruence subgroup containing the diagonal matrices in $\GL_2(A).$ Let $\Gamma_2=\{\gamma\in \Gamma: \det(\gamma)\in (\bbF_q^{\times})^2\}.$ We have an isomorphism
			$M(\Gamma)\cong M(\Gamma_2)$
			with \[M_{k,l}(\Gamma_2)=M_{k,l_1}(\Gamma)\oplus M_{k,l_2}(\Gamma)\] on each graded piece, where $l_1,l_2$ are the two solutions to $k\equiv 2l\pmod{q-1}.$ 
		\end{theorem}
		
		Since there are many intermediate lemmata involved, we break the proof of Theorem \ref{thm: decomp of mod forms} up into the next few parts of this section. We state and prove the generalization afterwards.
		
		\subsection{Properties of $\Gamma_2$}
		
		We begin with some group theory and elementary number theory which inspired our second main result and is instrumental in its proof.
		\begin{lemma}\label{l: Gamma2 normal, index 2, coset rep}
			Let $\Gamma \leq GL_2(A)$ be a congruence subgroup containing the diagonal matrices in $\GL_2(A).$ Let $\Gamma_2=\{\gamma\in \Gamma: (\det \gamma)\in (\bbF_q^{\times})^2 \}.$ This $\Gamma_2$ is a normal subgroup of $\Gamma$ with $[\Gamma:\Gamma_2]=2,$ and for any $\alpha\in \bbF_q^{\times}\setminus(\bbF_q^{\times})^2,$ the matrix $\begin{psmallmatrix}\alpha&0\\0&1\end{psmallmatrix}$ is a representative for the unique non-trivial left coset of $\Gamma_2$ in $\Gamma.$
		\end{lemma}
		\begin{proof}
			Let $\varphi:\Gamma\to \bbF_q^{\times}$ be the map $\gamma\mapsto (\det\gamma)^{(q-1)/2}.$ Since $(\det\gamma)^{q-1}=1$ for all $\gamma\in \Gamma,$ we see $\ker\varphi=\Gamma_2.$ If $\gamma\in \Gamma\setminus\Gamma_2$ then $(\det\gamma)^{(q-1)/2}= -1$ so $\varphi(\Gamma)\cong \bbZ/2\bbZ$ as multiplicative groups and $[\Gamma:\Gamma_2]=2.$\\
			
			If $\gamma\in \Gamma\setminus \Gamma_2,$ i.e.\ $\det(\gamma)\in \bbF_q^{\times}\setminus(\bbF_q^{\times})^2,$ then for any $\alpha\in \bbF_q^{\times}\setminus(\bbF_q^{\times})^2$ there is some $\gamma_2\in \Gamma_2$ with 
			\[\gamma=\begin{psmallmatrix}\alpha&0\\0&1\end{psmallmatrix}\gamma_2.\]
		\end{proof}
		
		
		We recall from elementary number theory the following.
		\begin{lemma}
		\label{l: elementary number theory}
			Suppose $q$ is odd. Integers $k,$ and $l$ satisfy $k\equiv 2l\pmod{q-1}$ if and only if \[l\equiv \begin{cases}\frac{k}{2}\pmod{q-1}, &\text{ or }\\ \frac{k}{2}+\frac{q-1}{2}\pmod{q-1}.\end{cases}\] 
		\end{lemma}
		\begin{proof}
			We know that $2l\equiv k\pmod{q-1}$ if and only if $2l-m(q-1)=k$ for some integer $m.$ If $\gcd(2,q-1)$ does not divide $k$ then there are no solutions, and if it does then there are exactly $\gcd(2,q-1)=2$ distinct solutions modulo $q-1.$ To be explicit, we illustrate this with computations:
			\begin{itemize}
				\item[$(\Rightarrow)$] Suppose that $k=m(q-1)+2l$ for some integer $m.$ Since $q-1$ is even, $k$ is even and $l=-m(\frac{q-1}{2})+\frac{k}{2}$ so $l\equiv\frac{k}{2}\pmod{\frac{q-1}{2}}.$ If $m$ is even, $\frac{m}{2}$ is an integer, and otherwise $\frac{m-1}{2}$ is, so we have
				\[
				l = \begin{cases}
						l_1\equiv \frac{k}{2}\pmod{q-1}, & m\text{ even}\\
						l_2\equiv \frac{k}{2}+\frac{q-1}{2}\pmod{q-1}, & m\text{ odd.}
					\end{cases}\]
				
				\item[$(\Leftarrow)$] Suppose $l=l_1\equiv \frac{k}{2}\pmod{q-1}.$ We have $l_1=n_1(q-1)+\frac{k}{2}$ for some $n_1,$ so $k=-2n_1(q-1)+2l_1.$ If $l=l_2\equiv \frac{k}{2}+\frac{q-1}{2}\pmod{q-1}$ then $l_2=n_2(q-1)+\frac{k}{2}+\frac{q-1}{2}$ for some $n_2$ and we have $k=-(2n_2+1)(q-1)+2l_2.$ In either case we conclude that $k\equiv 2l\pmod{q-1}.$
			\end{itemize}
		\end{proof}

		\subsection{Cusps and Elliptic Points}
		
		We wish to compare the cusps and elliptic points on the Drinfeld modular curves for $\Gamma$ and $\Gamma_2.$ As our notion of elliptic point is slightly different from Gekeler's, so that it adapts to the notion of a stacky Drinfeld modular curve more naturally, we discuss some of the properties of elliptic points with the next two group-theoretic results. 
		\begin{lemma}\label{l: stabilizers are conjugate}
				Let $\Gamma\leq \GL_2(A)$ be a congruence subgroup containing the diagonal matrices in $\GL_2(A).$
				If $e_1$ and $e_2$ are elliptic points for $\Gamma,$ then the stabilizers $\Gamma_{e_1}$ and $\Gamma_{e_2}$ are $\GL_2(A)$-conjugate. 
		\end{lemma}
		\begin{proof}
			Both $\Gamma_{e_1}$ and $\Gamma_{e_2}$ stricty contain $\bbF_q^{\times}$ by definition of an elliptic point, and each stabilizer is a subgroup of $\GL_2(A)_{e_i},$ for $i=1$ or $2.$ So, both elliptic points for $\Gamma$ are also elliptic points on $\Omega,$ i.e.\ they are lie in the same $\GL_2(A)$-orbit. It is well-known that stabilizers of any two elements in the same orbit are conjugate subgroups.
		\end{proof}
		
		\begin{lemma}\label{l: stabilizer index}
			Let $q$ be a power of an odd prime, let $\Gamma\leq \GL_2(A)$ be a congruence subgroup containing the diagonal matrices of $\GL_2(A).$ 
			Let $\Gamma_2=\{\gamma\in \Gamma:\det(\gamma)\in (\det\Gamma)^2\}.$ 
			Let $e\in \operatorname{Ell}(\Gamma_2).$ We have
			\[[\Gamma_e:(\Gamma_2)_e]=\begin{cases}1, & \text{if }\Gamma\text{ is ``square''}\\
			2, &\text{if }\Gamma\text{ is ``non-square.''}\end{cases}\]
		\end{lemma}
		\begin{proof}
			By definition, the stabilizer $\Gamma_e$ strictly contains $\bbF_q^{\times}$ and as this is a subgroup of the stabilizer $\GL_2(A)_e,$ we see that $e$ is an elliptic point for $\GL_2(A),$ i.e.\ an elliptic point on $\Omega.$ So, we know $\GL_2(A)_e\cong \bbF_{q^2}^{\times},$ which means $(\Gamma_2)_e\unlhd \Gamma_e\unlhd \GL_2(A)_e\cong \bbF_{q^2}^{\times}.$ Since 
			\[(\Gamma_2)_e=\ker((\det)^{\frac{q-1}{2}}:\Gamma_e\to \bbF_q^{\times}),\] the result is immediate according to whether $\displaystyle{(\det)^{\frac{q-1}{2}}}$ is surjective onto $\{\pm 1\}.$ That is, we need only check the ``parity'' of $\Gamma,$ i.e.\ whether $\Gamma_e$ contains some $\gamma$ with $\det\gamma\in \bbF_q^{\times}\setminus (\bbF_q^{\times})^2$ to determine the index of the stabilizer $(\Gamma_2)_e$ for all elliptic points $e.$
		\end{proof}
		
		The main idea for this step of the proof of Theorem \ref{thm: decomp of mod forms} is the following comparison between elliptic points and cusps for $\Gamma$ and $\Gamma_2.$
		\begin{proposition}\label{p: elliptic points and cusps}
			Let $q$ be a power of an odd prime, let $\Gamma\leq \GL_2(A)$ be a congruence subgroup containing the diagonal matrices of $\GL_2(A).$ Let $\Gamma_1=\{\gamma\in \Gamma:\det(\gamma)=1\}$ and let $\Gamma_2=\{\gamma\in \Gamma:\det(\gamma)\in (\det\Gamma)^2\}.$
			\begin{enumerate}
				\item $\operatorname{Ell}(\Gamma)=\operatorname{Ell}(\Gamma_2),$
				\item $\cC_{\Gamma}\subseteq\cC_{\Gamma_2}$
			\end{enumerate}
			Furthermore, if $\Gamma_1\leq \Gamma'\leq\Gamma$ for some congruence subgroup $\Gamma',$ then $\cC_{\Gamma}\subseteq\cC_{\Gamma'},$ i.e.\ the cusps of $\Gamma$ are some subset of the cusps of $\Gamma'$
		\end{proposition}
		\begin{proof}
			Suppose $e_2\in \operatorname{Ell}(\Gamma_2),$ so by definition the stabilizer $(\Gamma_2)_{e_2}$ is strictly larger than $\bbF_q^{\times}.$ Since $(\Gamma_2)_{e_2}$ is a subgroup of $\Gamma_{e_2},$ it must be that $\Gamma_{e_2}$ strictly contains $Z(\bbF_q),$ so $e_2\in \operatorname{Ell}(\Gamma),$ i.e.\ $\operatorname{Ell}(\Gamma_2)\subseteq \operatorname{Ell}(\Gamma).$\\
			
			For the same reason, if $e\in \operatorname{Ell}(\Gamma),$ then $e$ is an elliptic point on $\Omega,$ and we know $\GL_2(A)_e\cong \bbF_{q^2}^{\times}.$  In particular, as $\bbF_{q^2}^{\times}$ and $\bbF_q^{\times}$ are cyclic groups, we know $(\Gamma_2)_e$ and $\Gamma_e$ are cyclic and we have $1\unlhd Z(\bbF_q)\unlhd (\Gamma_2)_e\unlhd \Gamma_e\unlhd \GL_2(A)_e\cong \bbF_{q^2}^{\times}.$\\
			
			Since $q-1\mid \#\Gamma_e,$ there is some $n\mid q+1$ such that $\#\Gamma_e=n(q-1).$ 
			Suppose that $\langle \gamma \rangle =\Gamma_e.$ Since $Z(\bbF_q)\subset \Gamma_e,$ the subgroup $Z(\bbF_q),$ the unique subgroup of order $q-1$ in the cyclic group $\Gamma_e,$ is generated by $\gamma^n.$ So, a set of representatives of $\Gamma_e/\bbF_q^{\times}$ is $\{\gamma^0,\ldots, \gamma^n,\}$ and we write 
			\[\Gamma_e/\bbF_q^{\times}\cong \bbF_q^{\times}\oplus \gamma\bbF_q^{\times}\oplus\cdots\oplus\gamma^n\bbF_q^{\times}.\]
			
			
			We claim that if $\Gamma$ is ``non-square,'' the cosets with representatives 
			$\gamma^j$
			with $j$ even form a subgroup isomorphic to $(\Gamma_2)_e/\bbF_q^{\times}.$ 
			If $\Gamma$ is ``non-square'' then by Lemma \ref{l: stabilizer index} we know that $\Gamma_e$ contains some $\gamma'$ with $\det\gamma'$ a non-square, so $\det\gamma$ is non-square. Otherwise we would have $\gamma^n=\gamma'$ for some $n$ and with $\det\gamma\in (\bbF_q^{\times})^2,$ we would have $\det\gamma'$ a square which is a clear contradiction.
			For any even $n$ we have $\det(\gamma^n)$ a square.
			For odd $n,$ since $\det\gamma^n\in \bbF_q^{\times}\setminus(\bbF_q^{\times})^2$ then for any $\alpha'\in \bbF_q^{\times}$ non-square, there is some $\gamma_2\in \Gamma_2$ such that $\displaystyle{\gamma=\begin{psmallmatrix}(\alpha')^{1/n}&0\\0&1\end{psmallmatrix}\gamma_2}.$ However, whether this $(\alpha')^{1/n}$ is a square or not, 
			\[\det(\gamma^n)=\alpha'\det\gamma_2,\]
			which is not a square. Otherwise if $\Gamma$ is ``square,'' by Lemma \ref{l: stabilizer index} we have $\Gamma_e=(\Gamma_2)_e.$  
			Whether $\Gamma$ is square or not, $(\Gamma_e)/\bbF_q^{\times}$ has a nontrivial subgroup isomorphic to $(\Gamma_2)_e/\bbF_q^{\times},$ so the stabilizer of $e$ in $\Gamma_2$ strictly contains $\bbF_q^{\times}$ and $e\in \operatorname{Ell}(\Gamma_2).$ We have shown that $\operatorname{Ell}(\Gamma)\subseteq \operatorname{Ell}(\Gamma_2),$ completing the first part of the proof.\\

			Let $s\in \bbP^1(K).$ We know $\Gamma s\supseteq \Gamma_2 s,$ i.e.\ the action of $\Gamma_2$ partitions $\bbP^1(K)$ more finely than the action of $\Gamma.$ If $s_1,\cdots, s_n$ are cusps of $\Gamma,$ we write $\Gamma\setminus\bbP^1(K)=\Gamma s_1\sqcup \cdots \sqcup\Gamma s_n,$ and then 
			\[\Gamma s_i=\Gamma_2 s_i\sqcup (\Gamma\setminus \Gamma_2)s_i.\]
			If the points of $\bbP^1(K)$ in the orbits $(\Gamma\setminus \Gamma_2)s_i,$ under the action by $\Gamma_2$ have orbit representatives $t_1,\cdots, t_m$ then we can write 
			\[\Gamma_2\setminus\bbP^1(K)=\Gamma_2 s_1\sqcup \cdots \sqcup \Gamma_2 s_n\sqcup \Gamma_2t_1\sqcup\cdots\sqcup\Gamma_2 t_m,\] so the cusps of $\Gamma_2$ are $\cC_{\Gamma_2}=\{s_1,\cdots, s_n,t_1,\cdots, t_m\},$ which contains $\cC_{\Gamma}.$
			
			Finally, as we have made no reference to the particular choice $\Gamma'=\Gamma_2$ in our discussion of cusps, the last part of the proposition follows from this same argument.
\end{proof}

\subsection{Modularity and Series Expansions at Cusps}

Our next steps in the proof of Theorem \ref{thm: decomp of mod forms} deal with the $u$-series expansions of modular forms. 
\begin{proposition}\label{p: generator trick}
Let $f$ be holomorphic on $\Omega$ and at the cusps of $\Gamma_2,$ and let $\beta=\alpha^2\in \bbF_q^{\times},$ where $\alpha$ generates $\bbF_q^{\times}.$ If $f(\gamma z)=(\det\gamma)^{-l}(cz+d)^kf(z)$ for $\displaystyle{\gamma=\begin{psmallmatrix}a&b\\c&d\end{psmallmatrix}\in \Gamma_2},$ where $k/2$ is an integer, then \[f\left(\begin{psmallmatrix}\beta&0\\0&1\end{psmallmatrix}z \right) = f(\beta z)=\beta^{-k/2}f(z).\]
\end{proposition}
\begin{proof}
Since $\displaystyle{\begin{psmallmatrix}\alpha&0\\0&\alpha\end{psmallmatrix}\in \Gamma_2},$ for $f$ not identically zero we have $\displaystyle{f\left(\frac{\alpha z}{\alpha}\right) = f(z)=\alpha^{-2l}\alpha^k f(z)},$ and so 
$\alpha^{k-2l}=1.$\\

%

By assumption on $\beta,$ we know $\begin{psmallmatrix}\beta&0\\0&1\end{psmallmatrix}\in \Gamma_2$ and therefore $f(\beta z)=\beta^{-l}f(z).$ It suffices to show that $\beta^{-l}=\beta^{-k/2},$ that is, $\alpha^{-k}=\alpha^{-2l}.$ But, by Lemma \ref{l: weight-type}, if $f$ is not identically zero this follows from $k\equiv 2l\pmod{q-1}.$ Note that if $f$ is identically $0,$ the statement of the Proposition is trivial. 
\end{proof}

We complete the proof of Theorem \ref{thm: decomp of mod forms} with the following result.
\begin{proposition}
	Let $q$ be a power of an odd prime. Suppose $\Gamma$ is ``non-square.''
	Let $f$ be a modular form of weight $k$ and type $l$ for $\Gamma_2,$ where $k/2$ is an integer. 
	There are two modular forms $f_1$ and $f_2$ for $\Gamma$ of weight $k$ and types $l_1\equiv k/2\Mod{q-1}$ and $l_2\equiv k/2+(q-1)/2\Mod{q-1}$ respectively, such that $f=f_1+f_2.$
\end{proposition}
\begin{proof}
Suppose that $f(\gamma_2 z)=(\det\gamma_2)^{-l}(cz+d)^kf(z)$ for $\displaystyle{\gamma_2=\begin{psmallmatrix}a&b\\c&d\end{psmallmatrix}\in \Gamma_2}.$ 
Write the $u$-series $f(z)=\sum_{n\geq 0}a_nu^n.$ Let $\beta = \alpha^2\in \bbF_q^{\times},$ where $\alpha$ generates $\bbF_q^{\times}.$ By Proposition \ref{p: generator trick}, $f(\beta z)=\beta^{-k/2}f(z).$ Using this relationship, we have from Lemma \ref{l: u(a/d)=d/au}
\[f(\beta z)= \sum_{n\geq 0}a_n\beta^{-n}u^n = \beta^{-k/2}\left(\sum_{n\geq 0}a_nu^n\right), \] so for each non-zero $a_n$ we have $\beta^{-n}=\beta^{-k/2},$ that is, $\alpha^{-2n}=\alpha^{-k}.$ So, $k\equiv 2n\pmod{q-1},$ and by removing the zero summands from the $u$-series and using Lemma \ref{l: elementary number theory}, we may write
\[f(z)=\sum_{n\equiv k/2\Mod{q-1}}a_nu^n+\sum_{n\equiv k/2+(q-1)/2\Mod{q-1}} a_nu^n. \]


Let \[f_1\overset{def}{=}\sum_{n\equiv k/2\Mod{q-1}}a_nu^n\quad \text{ and }\quad f_2\overset{def}{=}\sum_{n\equiv k/2+(q-1)/2\Mod{q-1}} a_nu^n\] be the modular forms for $\Gamma_2$ uniquely determined by their $u$-series by Lemma \ref{l: u-series coeffs determine form}. Note that since $f(z)=(f_1+f_2)(z),$ we see immediately $f_1(z)$ and $f_2(z)$ are indeed modular forms for $\Gamma_2$: \[f(\gamma_2z)=(f_1+f_2)(\gamma_2 z)=(cz+d)^k\det\gamma_2^{-l}(f_1+f_2)(z);\] holomorphy follows since no holomorphic function ($f$ in this case) is the sum of non-holomorphic functions ($f_1$ and $f_2$); holomorphy at $\infty$ follows from definition of $f$.\\

Let $\alpha\in \bbF_q^{\times}$ be some non-square, so by Lemma \ref{l: u(a/d)=d/au} we have $u(\alpha z)=\alpha^{-1}u(z).$ We have \[f_1(\alpha z)=\sum_{n\equiv k/2\Mod{q-1}}a_n\alpha^{-n}u^n=\alpha^{-l_1}\sum_{n\equiv k/2\Mod{q-1}}a_nu^n,\] where $\displaystyle{l_1\equiv \frac{k}{2}\pmod{q-1}}$ by Lemma \ref{l: elementary number theory}. Let $\gamma\in \Gamma\setminus \Gamma_2.$ For any $\alpha\in \bbF_q^{\times}\setminus(\bbF_q^{\times})^2$ there is some $\displaystyle{\gamma_2=\begin{psmallmatrix}a&b\\c&d\end{psmallmatrix}\in \Gamma_2}$ such that \[\gamma=\begin{psmallmatrix}\alpha&0\\0&1\end{psmallmatrix}\gamma_2,\] so 
\[f_1(\gamma z)=f_1(\begin{psmallmatrix}\alpha&0\\0&1\end{psmallmatrix}\gamma_2 z)=\alpha^{-l}f_1(\gamma_2 z)=\alpha^{-l}\det(\gamma_2)^{-l}(cz+d)^kf_1(z)=\det(\gamma)^{-l}(cz+d)^kf_1(z)\] and $f_1$ is a modular form for $\Gamma.$ Likewise we have 
\[f_2(\alpha z)=\sum_{n\equiv k/2+(q-1)/2\Mod{q-1}} a_n\alpha^{-n}u^n=\alpha^{-l_2}\sum_{n\equiv k/2+(q-1)/2\Mod{q-1}} a_nu^n,\]
where now $\displaystyle{l_2\equiv \frac{k+q-1}{2}}\pmod{q-1}$ by Lemma \ref{l: elementary number theory}. So, for $\gamma,\alpha$ and $\gamma_2$ as above,
\[f_2(\gamma z)=\alpha^{-l}\det(\gamma_2)^{-l}(cz+d)^kf_2(z)\] and $f_2$ is a modular form for $\Gamma.$
\end{proof}

\subsection{Generalization}

We will show that Theorem \ref{thm: decomp of mod forms} is actually a special case of the following result.
\begin{theorem}
	\label{thm: generalized decomp}
	Let $q$ be a power of an odd prime. Let $\Gamma\leq \GL_2(A)$ be a congruence subgroup. Let $\Gamma_1=\{\gamma\in \Gamma: \det(\gamma)=1\}.$ Suppose that $\Gamma_1\leq \Gamma'\leq \Gamma$ for some congruence subgroup $\Gamma'.$ As algebras
	\[M(\Gamma)=M(\Gamma'),\] and each component $M_{k,l}(\Gamma')$ is some direct sum of components $M_{k,l'}(\Gamma)$ for some nontrivial $l',$ the distinct solutions to $k\equiv [\Gamma:\Gamma']l'\pmod{q-1},$ where $k/2$ is an integer.
\end{theorem}
\begin{remark}
	The subgroups $\Gamma'$ which appear in the statement of Theorem \ref{thm: generalized decomp} may be thought of as the inverse image under $\det:\Gamma\to \bbF_q^{\times}$ of some subgroup of $\bbF_q^{\times}.$ As $\bbF_q^{\times}$ is cyclic, every subgroup $H\leq \bbF_q^{\times}$ is normal, and hence each $\Gamma'=\det^{-1}(H)$ is normal in $\Gamma.$
\end{remark}
\begin{proof}(Theorem \ref{thm: generalized decomp})
	Write $f|_{\gamma}$ for the (Petersson) \textbf{slash operator} of weight $k$ and type $l$ for $\displaystyle{\gamma=\begin{psmallmatrix}a&b\\c&d\end{psmallmatrix}\in \GL_2(K)}$ defined by \[f|_{\gamma}\overset{def}{=}\det(\gamma)^l(cz+d)^{-k}f(\gamma z).\] If $f\in M_{k,l}(\Gamma'),$ by normality $\Gamma'\unlhd \Gamma$ we have that $f|_{\gamma}$ is weakly modular of weight $k$ and type $l$ for $\Gamma,$ that is, for any $\gamma\in \Gamma.$ By Proposition \ref{p: elliptic points and cusps} we see that $f|_{\gamma}$ is holomorphic at the cusps of $\Gamma$ since $f$ is holomorphic at the cusps of $\Gamma',$ indeed we claim the $u$-series of expansions of $f|_{\gamma}$ and $f$ agree at the cusps of $\Gamma'.$\\
	
	We consider the $u$-expansions in a small neighborhood of each cusp. The action of $\Gamma'$ on a neighrborhood of a cusp is trivial, so by the Third Isomorphism Theorem for Groups since $\Gamma_1\leq \Gamma'$ we have an action of the finite group \[\Gamma/\Gamma'\cong\det(\Gamma)/\det(\Gamma'),\] which has order some divisor of $q-1$ since $\{1\}=\det \Gamma_1\leq \det\Gamma'\leq \det\Gamma\leq \bbF_q^{\times}.$ We may describe the group ring $\bbF_q[\Gamma/\Gamma']$ via idempotents as follows. Let $n'\overset{def}{=}\#(\det\Gamma')$ and let $n\overset{def}{=}\#(\det\Gamma).$ This means 
	\[\bbF_q[\Gamma/\Gamma']=\bigoplus_{i=0}^{n/n'-1}\bbF_qe_i,\]
	where $\Gamma$ acts on the $e_i$ via maps $\gamma\mapsto (\det\gamma)^{in'}.$ So as $\Gamma$-modules, we have \[M_{k,l}(\Gamma')=\bigoplus_i M_{k,l}(\Gamma')e_i,\] where \[M_{k,l}(\Gamma')e_i=M_{k,l+in'}(\Gamma).\] Finally, since modular forms for $\Gamma'$ are holomorphic at the cusps of $\Gamma',$ and by Proposition \ref{p: elliptic points and cusps} the cusps of $\Gamma$ are a subset of the cusps of $\Gamma',$ we know $\Gamma'$-modular forms are holomorphic at the cusps of $\Gamma.$
\end{proof}

\begin{remark}
	One can verify that the slash operators $f|_{\gamma}$ are holomorphic at the cusps of $\Gamma$ directly by considering their $u$-series expansions at small neighborhoods of the cusps of $\Gamma.$
\end{remark}

\begin{remark}
	Theorem \ref{thm: decomp of mod forms} is just the special case of Theorem \ref{thm: generalized decomp} when $\Gamma'=\Gamma_2.$ We highlight the special case Theorem \ref{thm: decomp of mod forms} in this article because of its relationship with the other main result Theorem \ref{thm: forms to differentials}.
\end{remark}

\subsection{Summary}
Our first result in this section, Theorem \ref{thm: forms to differentials}, tell us about the geometry of Drinfeld modular forms for congruence subgroups consisting of matrices with square determinants.\\ 

According to whether a given congruence subgroup $\Gamma$ is ``square'' or not, we can decompose the algebra of Drinfeld modular forms for $\Gamma_2$ with Theorem \ref{thm: decomp of mod forms}.
We have shown a modular form $f$ of weight $k$ and type $l$ for $\Gamma_2$ is holomorphic at the cusps of $\Gamma,$ and there are two choices of type $l_1$ and $l_2,$ the lifts of a given $\displaystyle{l\equiv k/2\Mod{\frac{q-1}{2}}}$ to $\bbZ/(q-1)\bbZ$ such that $f$ may be a sum of modular forms of weight $k$ and type either $l_1$ or $l_2$ for $\Gamma,$ if $\Gamma$ is ``non-square.'' Together, these conditions are the definition of a modular form, so every modular form for $\Gamma_2$ is associated to a pair of $\Gamma$ modular forms in this case. Finally, we generalize this decomposition with Theorem \ref{thm: generalized decomp}.

\section{Application: Computing Algebras of Drinfeld Modular Forms}

We have seen some relationships between the modular forms for $\Gamma$ and $\Gamma_2.$ Now we apply the geometry of those modular forms, in particular in terms of stacks. We conclude the article with some examples of how we intend to use geometric invariants to compute algebras of Drinfeld modular forms. Our examples of Drinfeld modular curves have genus $0$ since we need techniques for $\bbQ$-divisors such as in \cite{ODorney-canonical-rings-Q-divisors-on-P1} to compute the canonical ring of a Drinfeld modular curve. In particular, the results in \cite{VZB} do not apply to log curves where the cuspidal divisor $\Delta$ has coefficients besides $1$ or is supported at stacky points. Generalizations of both \cite{VZB} for all divisors with $\bbQ$-coefficients and \cite{ODorney-canonical-rings-Q-divisors-on-P1} to higher genera cases are limited to \cite{LRZ} and \cite{Cerchia-Franklin-ODorney-Qdiv-Ell-curves} and which we do not use here.\\

As a first example, we will recall the computation of canonical rings for classical modular curves in \cite{VZB}, and see how this differs from the case of a Drinfeld modular curve, as we demonstrate that $M((\GL_2(A))_2)=C[g,h].$ Since $\GL_2(A)$ is ``non-square,'' we first reduce to $\GL_2(A)_2$ according to Theorem \ref{thm: decomp of mod forms}. We use the geometry of modular forms for this smaller group, Theorem \ref{thm: forms to differentials}, and geometric invariants of the modular curve for $\GL_2(A)_2$ to compute a canonical ring as in \cite{VZB}.\\

In the classical setting, modular curves are tame stacky curves, and likewise in the Drinfeld setting, as we will explain next. As in \cite[Definition $(3.5)$]{Gekeler-survey-Drinfeld-modular-forms}, for a non-identically-zero Drinfeld modular form $f$ of weight $k$ and type $l$ for $\GL_2(A),$ we let $v_z(f)$ denote the vanishing order of $f$ at $z\in \Omega$ and $v_{\infty}(f)$ denote the vanishing order of $f$ at $\infty.$ From \cite[Equation $(3.10)$]{Gekeler-survey-Drinfeld-modular-forms}, for such an $f$ we have the following \textbf{valence formula}: 
\[\sideset{}{^*}\sum_{z\in GL_2(A)\setminus \Omega}v_z(f)+\frac{v_e(f)}{q+1}+\frac{v_{\infty}(f)}{q-1}=\frac{k}{q^2-1},\]
where $\sum^*$ denotes a sum over non-ellitic classes of $\GL_2(A)\setminus \Omega.$ In particular, the characteristic of $C$ does not divide the degree of the stabilizer of any point on a Drinfeld modular curve, as $\deg G_x$ divides $q^2-1$ for all points $x.$\\

Next we turn to cusps of modular curves, where the classical and Drinfeld settings begin to differ. 
Both in the classical and Drinfeld settings, cusps of a modular curve are stable under M\"obius transformations by diagonal matrices. However, whereas a classical modular curve is a quotient of the upper half-plane $\cH=\{z=a+bi\in \bbC: b>0\}$ by a congruence subgroup of $\SL_2(\bbZ),$ 
so the cusps of a classical modular curve are not stacky points, in the Drinfeld setting the divisor of cusps of a modular curve should be regarded as an effective divisor which is a formal sum of distinct stacky points. Indeed, diagonal matrices in $\GL_2(A)$ have determinants in $\bbF_q^{\times}$ as opposed to determinant $1$ in the classical case of $\SL_2(\bbZ)$ acting on the upper half-plane $\cH$ of $\bbC,$ so a log divisor in the Drinfeld setting may have coefficients besides $1.$\\

We can compute section rings for general $\bbQ$-divisors on genus $0$ curves using \cite{ODorney-canonical-rings-Q-divisors-on-P1}, so 
we will consider the coefficients of log canonical divisors for Drinfeld modular curves in more detail. In particular, we compare the stabilizers of stacky points for $\GL_2(A)$ and $\GL_2(A)_2,$ since we use these to write down log canonical divisors for the stacky curves associated with these groups.\\ 

Recall the parameter at $\infty$ in the Drinfeld setting, introduced in Definition \ref{d: parameter at infty}. Since $\displaystyle{\frac{de_A(z)}{dz}=1},$ we have 
$du = -\overline{\pi}u^2 dz,$ so the differential $dz$ in this case has a double pole at $\infty.$ But, $\infty$ is stabilized by upper triangular matrices in $\GL_2(A).$ As the group of upper triangular matrices is strictly larger than $\displaystyle{\#\left\{\begin{psmallmatrix}\alpha&0\\0&\alpha\end{psmallmatrix}:\alpha\in \bbF_q^{\times} \right\}}=q-1,$ the point $\infty$ is an elliptic point of $\GL_2(A)$ and hence a stacky point for $\sX.$ In fact, both stacky points on $\sX_{\GL_2(A)_2}$, the unique $e$ on $\Omega$ from Definition \ref{d: elliptic pt} and $\infty$ have stabilizers half the order of their stabilizers in $\GL_2(A),$ which comes from the double cover $\Spec C[\tilde{\j}]\to \Spec C[j]$ (see \cite[Page $312$]{Breuer-Gekeler-h-function}) which is ramified above $j=0$ and $\infty$ and the fact that $\GL_2(A)$ is ``non-square'' so that $[\GL_2(A)_e:(\GL_2(A)_2)_e]=2$ by Lemma \ref{l: stabilizer index}.\\


We summarize the above and calculate a log canonical ring in the following example.
\begin{example}
Let $\sX$ be the Drinfeld modular curve with coarse space $X$ whose analytification is $X^{\text{an}}=\GL_2(A)_2\setminus(\Omega\cup\bbP^1(K)).$ This $\sX$ is a stacky $\bbP^1$ with two stacky points: 
\begin{itemize}
	\item a point $P_e$ with a stabilizer of order $\displaystyle{\frac{q+1}{2}}$ corresponding to the unique elliptic point of $\Omega$ (note that $\GL_2(A)$ is ``non-square'')
	\item a cusp, denoted $\infty,$ with a stabilizer of order $\displaystyle{\frac{q-1}{2}}.$
\end{itemize} 

Let \[D=K_{\sX}+2\Delta\sim K_{\bbP^1}+\left(1-\frac{1}{\frac{q+1}{2}} \right)P_e + \left(1+\frac{1}{\frac{q-1}{2}}\right)\infty+2\infty\] be a log stacky canonical divisor on $\sX.$ By \cite[Theorem $6$]{ODorney-canonical-rings-Q-divisors-on-P1} we may construct generators for our canonical ring by using best upper and lower approximations to the coefficients of our log canonical divisor, and laboriously doing so we have \[R_D\cong C[g,h]\cong M(\GL_2(A)_2).\]
\end{example}

\begin{remark}\label{r: how we compute canonical rings in our examples}
	Our example computations of presentations for algebras of Drinfeld modular forms are not the most direct means of obtaining such presentations, nor are our examples new. We simply show a solution to Gekeler's problem by using geometric invariants. We turn the problem of presenting an algebra of modular forms into a study of Riemann-Roch spaces where we find our example generators and relations by considering best approximations to coefficients of a log canonical divisor as in \cite{ODorney-canonical-rings-Q-divisors-on-P1}. We defer a thorough description of the technique to O'Dorney's article and content ourselves with the remarks:
	\begin{enumerate}
		\item we can determine a presentation for the section ring of any $\bbQ$-divisor on a curve of genus $g\leq 1$ using \cite{ODorney-canonical-rings-Q-divisors-on-P1}, its generalization \cite{Cerchia-Franklin-ODorney-Qdiv-Ell-curves}, and \cite{VZB};
		\item constructive theories of log canonical rings for other curves are found in \cite{Landesman-Ruhm-Zhang-Spin-canonical-rings} and especially in \cite{VZB};
		\item the best-approximation technique we use here is laborious but straightforward to use, though for many examples it does not give aesthetic results (\cite{VZB} covers all of the nice cases).  
	\end{enumerate}
\end{remark}

To conclude, we present some new examples of computations of algebras of Drinfeld modular forms. The idea is to illustrate the role of our theory in this calculation, and the limited scope of our results now indicates a clear direction for future work.\\     

As with our previous example, since the existing theory is most developed in genus $0,$ we begin by seeking Drinfeld modular curves in genus $0.$ We know from \cite[Theorem $8.1$]{Gekeler-Invariants} genus formulae for the modular curves associated to $\Gamma(N), \Gamma_1(N)$ and $\Gamma_0(N),$ where we recall \[\Gamma_1(N)=\left\{\left(\begin{array}{cc}1&*\\0&*\end{array}\right)\pmod{N}\right\}\text{ and } \Gamma_0(N)=\left\{\left(\begin{array}{cc}*&*\\0&*\end{array}\right)\pmod{N}\right\}.\] If $\deg N>1,$ then $g(X(N))>0,$ so we consider the case of linear level. Cornelissen has two papers \cite{Cornelissen-lvlT} and \cite{Cornelissen-wt1} dedicated to the Drinfeld modular forms for $\Gamma(\alpha T+\beta),$ for $\alpha,\beta\in \bbF_q$ and $\alpha\neq 0.$ We consider
$M(\Gamma_1(T+\beta)) $ and $M(\Gamma_0(T+\beta))$ to be consistent with our description of monic level.
In fact, we know from \cite[Theorem $4.4$]{Dalal-Kumar-Gamma_0(T)-structure} that for $R$ any ring such that $A\subset R\subset C,$ the $R$-algebra of Drinfeld modular forms $M(\Gamma_0(T))_R$ with coefficients in $R$) is generated by $E_T(z)$ (from Example \ref{example: Eisenstein series for Gamma0(T)}, and the Drinfeld modular forms
\[\Delta_T(z)\overset{def}{=}\frac{g(Tz)-g(z)}{T^q-T} \text{ and }\Delta_W(z)\overset{def}{=}\frac{T^qg(Tz)-Tg(z)}{T^q-T}\]
for $\Gamma_0(T)$ (from \cite[Equation $(4.1)$]{Dalal-Kumar-Gamma_0(T)-structure}). Furthermore, \cite[Theorem $4.4$]{Dalal-Kumar-Gamma_0(T)-structure} tell us that the surjective map 
$R[U,V,Z]\to M(\Gamma_0(T))_R$ defined by $U\to \Delta_W,$ $V\to \Delta_T$ and $Z\to E_T$ induces an isomorphism 
\[R[U,V,Z]/(UV-Z^{q-1})\cong M(\Gamma_0(T))_R.\] Note that from \cite[Proposition $4.3(3)$]{Dalal-Kumar-Gamma_0(T)-structure} we know that $M_{k,l}(\Gamma_0(T))$ has an integral basis, i.e.\ a basis consisting of modular forms with coefficients in $A.$\\

To apply Theorem \ref{thm: forms to differentials} we need the following geometric invariants of 
$\sX(\Gamma_0(\alpha T+\beta))$:
\[\begin{array}{c|c|c}
\text{Genera}&\text{Elliptic points}&\text{Cusps}\\
\hline
\hline
\begin{array}{c}g(\sX(\Gamma_0(\alpha T+\beta)))=0\\\text{(by \cite[Thm $8.1.(\mathrm{iii})$]{Gekeler-Invariants})}\end{array}&\begin{array}{c}\#\operatorname{Ell}(\Gamma_0(\alpha T+\beta))\geq1\\	\text{ramified with index }q+1\text{ over }X(1)\\\text{(\cite[Proposition $7.3$]{Gekeler-Invariants})}\end{array}&\begin{array}{c}\#\cC_{\Gamma_0(\alpha T+\beta)}=2\\\text{(\cite[Proposition $6.7.(\mathrm{i})$]{Gekeler-Invariants})}\end{array}\\
\hline
\hline
g(\sX(\Gamma_0(\alpha T+\beta)_2)) &\begin{array}{c}\#\operatorname{Ell}(\Gamma_0(\alpha T+\beta)_2)=\#\operatorname{Ell}(\Gamma_0(\alpha T+\beta))\\\text{(by Proposition \ref{p: elliptic points and cusps})}\end{array}&\#\cC_{\Gamma_0(\alpha T+\beta)_2}\\
\hline
\end{array}\]

Recall that from \cite[Section $4$]{Dalal-Kumar-Gamma_0(T)-structure} we know the only two cusps of $\Gamma_0(T),$ which we write $0$ and $\infty,$ are exchanged by the matrix 
\[W_T\overset{def}{=}\left(\begin{array}{cc}0&-1\\T&0\end{array}\right).\]

While \cite[Proposition $7.2$]{Gekeler-Invariants} and \cite[Proposition $7.3$]{Gekeler-Invariants} give us some way to compute the number of elliptic points, in particular Gekeler's definition of an elliptic point (\cite[$(3.2)$]{Gekeler-Invariants} - the class of an elliptic point on $\Omega$ in $Y_{\Gamma}$) is slightly different from ours (recall Definition \ref{d: elliptic pt}). For our calculation to work, we must consider all points from $\Omega$ on 
$X_{\Gamma_0(\alpha T+\beta)}$ whose stabilizers under 
$\Gamma_0(\alpha T+\beta)$ 
strictly contain $\bbF_q^{\times}.$ Furthermore, we need to know the order of the stabilizers of each elliptic point for 
$\Gamma_0(\alpha T+\beta)_2,$ which depends on whether the congruence subgroup 
$\Gamma_0(\alpha T+\beta)$ is ``square,'' by Proposition \ref{p: elliptic points and cusps}.\\

\begin{conjecture}	
		Both $\Gamma_1(\alpha T+\beta)$ and $\Gamma_0(\alpha T+\beta)$ may be ``non-square'' congruence subgroups for any choice of $\alpha\neq 0$ and $\beta.$ 
\end{conjecture}
\begin{remark}
	We see this explictly for sufficiently small $q$ by means of the following algorithm:
	\begin{enumerate}
		\item Fix a level $N=\alpha T+\beta$ for $\alpha\in \bbF_q^{\times}$ and $\beta\in \bbF_q.$
		\item For all $a,b,c,d\in \bbF_q,$ compute the polynomial $(aN+1)d-bcN.$ If this is an element of $\bbF_q^{\times},$ then $\displaystyle{\gamma\overset{def}{=}\left(\begin{array}{cc} aN+1&b\\cN&d\end{array}\right)\in \Gamma_1(\alpha T+\beta)}$ (so $\gamma\in \Gamma_0(\alpha T+\beta)$ as well).
		\item If $c\neq 0$ and the polynomial $\displaystyle{z^2+\frac{d-(aN+1)}{cN}z-\frac{b}{cN}}$ is irreducible over $K_{\infty},$ then we know $\gamma\in \Gamma_i(\alpha T+\beta)_e\setminus \bbF_q^{\times}$ is a non-trivial stabilizer of the elliptic point. 
	\end{enumerate}
	This irreducibility condition follows from the fact that there exists some $z\in \Omega = C-K_{\infty}$ such that we have 
	\[\frac{az+b}{cz+d}=z \iff \begin{cases}
	a=d \text{ and } b=0=c, &\text{ or }\\
	cz^2+(d-a)z-b=0 & \text{ is irreducible over }K_{\infty},
	\end{cases}\] since if this polynomial in $A[z]$ had a solution, it would be an element of $K,$ which is a contradiction to our definition of $z.$\\
	
	As long as $q\leq 25,$ it suffices to check only the constants in $A$ by brute force to see that we can not only find such irreducible polynomials, but we can do so in such a way that every element of $\bbF_q^{\times}$ is the determinant of some matrix in the stabilizer $\Gamma_1(\alpha T+\beta)_e$ and hence also $\Gamma_0(\alpha T+\beta)_e.$ Therefore, it seems plausible that we will be able to find stabilizing matrices with non-square determinants for any odd $q.$  
\end{remark}
\begin{example}
	For completeness, we illustrate a non-trivial, non-square matrix in $\Gamma_1(4T+3)_e$ in the case $q=7.$ Consider $\displaystyle{\gamma=\left(\begin{array}{cc}4T+4&1\\5T+2&3\end{array}\right)}.$ Since $4T+4\equiv 1\pmod{4T+3}$ and $5T+2\equiv 0\pmod{4T+3}$ (since $5T+2=3(4T+3)$) and Sagemath tells us the corresponding polynomial $\displaystyle{z^2+\frac{2T+4}{T+6}z+\frac{4}{T+6}}$ is irreducible. Note that $\det\gamma = 3,$ which is not a square in $\bbF_7.$
\end{example}


Gekeler explains multiple techniques for computing the genus of a Drinfeld modular curve in \cite[Section $8$]{Gekeler-Invariants}, but to avoid repeating too much of his notation, we describe just two:
\begin{enumerate}
	\item Compute fibers of the ramified graph coverings $\Gamma\setminus \sT\to GL_2(A)\setminus \sT,$ where $\sT$ is the Bruhat-Tits tree as in \cite{Gekeler-Nonnengardt-BruhatTitsTrees} and \cite{Franklin-Ho-Papikian-DrinfeldCurves-SL}
	
	\item Riemann-Hurwitz formula and \cite[Proposition $8.3$]{Gekeler-Invariants}.
\end{enumerate}
The content of \cite[Proposition $8.3$]{Gekeler-Invariants} is that canonical coverings of Drinfeld modular curves have the least cuspidal ramification allowed by the group structure of the stabilizers of cusps and the only ramification possible is at elliptic points or cusps. 
This theory applies to the covers \[\sX(\Gamma_1(\alpha T+\beta)_2))\to \sX(\Gamma_1(\alpha T+\beta))\text{ and }\sX(\Gamma_0(\alpha T+\beta)_2)\to \sX(\Gamma_0(\alpha T+\beta)),\] which are canonical in the sense that by the universal property of pull-backs there are maps 
\[\psi_i:\sX(\Gamma_i(\alpha T+\beta)_2)\to 
\sX(\Gamma_i(\alpha T+\beta))\times_{\sX(\GL_2(A))}\sX(\GL_2(A)_2),\]
for $i=0, 1$ and if we compose $\psi_i$ with the canonical projection from the fiber-product onto $\sX(\Gamma_i(\alpha T+\beta)),$ we have a cover. Finally, all of the cusps of $X(N)$ are $\Gal(X(N)/\GL_2(A))$-conjugate, so if we consider $x=\infty$ in particular, and denote its stabilizer \[G_{\infty}=\Gamma_i(\alpha T+\beta)_{\infty}/\bbF_q^{\times},\] the first ramification group $G_{\infty, 1}$ is its $p$-Sylow subgroup $U_i(\alpha T+\beta)\cdot \bbF_q^{\times}/\bbF_q^{\times},$ where \[U_i(\alpha T+\beta)=\left\{\left(\begin{array}{cc}1&b\\0&1\end{array}\right)\in \Gamma_i(\alpha T+\beta)\right\}.\]

%
%

We conclude with one final example, where we will use $M(\Gamma_0(T))=C[U,V,Z]/(UV-Z^2)$ from \cite[Theorem $4.4$]{Dalal-Kumar-Gamma_0(T)-structure} to make sure that the log stacky canonical ring of the corresponding Drinfeld modular curve $\sX_{\Gamma_0(T)_2}$ does in fact compute this algebra of Drinfeld modular forms for $\Gamma_0(T)_2.$ We explicitly use \cite[Theorem $6$]{ODorney-canonical-rings-Q-divisors-on-P1}, demonstrating the best approximation technique discuseed in Remark \ref{r: how we compute canonical rings in our examples}.
\begin{example}
	Since $UV-Z^2$ describes a conic, we know that the curve $C[U,V,Z]/(UV-Z^2)\subset\bbP^2_C$ is rational, and all rational curves have genus $0.$ There are $2$ cusps, say $0$ and $\infty$ for $\sX_{\Gamma_0(T)}$ so there are at least the same cusps on $\sX_{\Gamma_0(T)_2}$ and hence there are $2$ elliptic points.\\ 
	
	Let $\overline{\Gamma_0(T)_2}$ denote the image of $\Gamma_0(T)_2$ in $\GL_2(A/T)\cong \GL_2(\bbF_q).$ As in \cite[Section $3$]{Gekeler-Invariants}, let $(A/T)^2_{\text{prim}}$ denote the primitive vectors in $A/T\times A/T,$ i.e.\ those vectors which span a non-zero direct summand. From \cite[Section $3$]{Gekeler-Invariants} we know 
	\[\{\text{cusps of }X_{\Gamma_0(T)_2}\}\cong \overline{\Gamma_0(T)_2}\setminus (A/T)^2_{\text{prim}}/\bbF_q^{\times},\] so the cusps of $X_{\Gamma_0(T)_2}$ are precisely the $\Gamma_0(T)_2$-orbits of $0$ and $\infty$ which correspond to the primitve vectors $(1,0)$ and $(0,1).$ So, there are exactly these two cusps and no further elliptic points. Let $\alpha\in \bbQ$ be such that \[\frac{2k-2l-kq}{k(q-1)}\leq \alpha<\frac{2k-2l-kq}{k(q-1)}+1\] and the number $r$ of best lower approximations to $\alpha$ with denominator strictly greater than $1$ is $r=2.$ Let
	\begin{align*}
		D\overset{def}{=}K_{\sX_{\Gamma_0(T)_2}}+2\Delta&\sim K_{\bbP^1}+\alpha(0)+\alpha(\infty)+2(0+\infty)\\
		&=\alpha(\infty)+(\alpha+2)(0),
	\end{align*}
	since $K_{\bbP^1}= -2\infty.$ We see that 
	\begin{align*}
		h^0\left(\frac{k}{2}D\right)&=2\Big\lfloor\frac{k}{2}(\alpha)\Big\rfloor+k+1\\
		&=k\left(\frac{2k-2l-kq}{k(q-1)}\right)+k+1\\
		&=1+\frac{k-2l}{q-1}\\
		&=\dim_C(M_{k,l}(\Gamma_0(T))),
	\end{align*}
	where we know this dimension from \cite[Proposition $4.1$]{Dalal-Kumar-Gamma_0(T)-structure}.\\
	
	Finally, we see from \cite[Theorem $6$]{ODorney-canonical-rings-Q-divisors-on-P1} that the canonical ring $R_D,$ i.e.\ the log stacky canonical ring for $\sX_{\Gamma_0(T)_2},$ is generated by $3$ functions: $\Delta_T,$ $\Delta_W$ and $E_T,$ corresponding to $U,V$ and $Z$ respectively, and has a single relation $UV-Z^2.$ 
\end{example}

\newpage
\bibliographystyle{amsalpha} 								
\bibliography{bibliography} 						
\end{document}